\documentclass[a4paper, 10 pt]{elsarticle}
\usepackage{amsmath,amssymb,amsthm}
\usepackage{hyperref}

\journal{Journal of \LaTeX\ Templates}
\nonstopmode \numberwithin{equation}{section}
\newtheorem{theorem}{Theorem}[section]
\newtheorem{lemma}[theorem]{Lemma}
\newtheorem{proposition}[theorem]{Proposition}
\newtheorem{remark}[theorem]{Remark}
\newtheorem{corollary}[theorem]{Corollary}
\newtheorem{definition}[theorem]{Definition}

\newcommand{\norm}[1]{\left\|#1\right\|}
\newcommand{\abs}[1]{\left\lvert#1\right\rvert}
\usepackage{xcolor}
\usepackage{tcolorbox}
\usepackage{multicol}
\newtcolorbox[auto counter,number within=chapter]{prb}{%
	colback=green!10!white,colframe=green!20!black,fonttitle=\bfseries}
\begin{document}
	
	\begin{frontmatter}
		
		\title{Approximate Controllability of Stochastic Hemivariational Control problem in Hilbert spaces }
		
		\author{Bholanath Kumbhakar\fnref{myfootnote}}
		\address{Department of Mathematics, Indian Institute of Technology Roorkee}
		\fntext[myfootnote]{Email id: bkumbhakar@mt.iitr.ac.in}
		
		\author{Deeksha\fnref{myfootnote}}
		\address{Department of Mathematics, Indian Institute of Technology Roorkee}
		\fntext[myfootnote]{Email id: deeksha\_s.ma.iitr.ac.in}
		
		
		\author{Dwijendra Narain Pandey\corref{mycorrespondingauthor}}
		\cortext[mycorrespondingauthor]{Corresponding author}
		\address{Department of Mathematics, Indian Institute of Technology Roorkee}
		\ead{dwij@ma.iitr.ac.in}
		
		
		\begin{abstract}
			In this paper, we discuss the approximate controllability for control systems governed by stochastic evolution hemivariational inequalities in Hilbert spaces. The interest in studying this type of equation comes from its application in some physical models like the thermostat temperature control or the diffusion through semi-permeable walls. Firstly, we introduce the concept of weak solutions for hemivariational inequalities. Then, the controllability is formulated by utilizing stochastic analysis techniques and properties of Clarke subdifferential operators, as well as applying multivalued fixed point theorem. Finally, we conclude this article by an application in stochastic heat propagation problem.
		\end{abstract}
		
		
	\end{frontmatter}

	\section{Introduction}
	Control theory finds applications across various fields in both engineering and mathematics. In fact, it plays a crucial role in the functioning of many modern devices, quietly operating behind the scenes in computers, automobiles, and other everyday technologies. As noted by Cara and Zuazua \cite{zuazua2003control}, control theory serves as a meeting point for numerous mathematical concepts and methods, giving rise to a rich and evolving area of mathematics. In the mathematical context, control theory is a branch of applied mathematics that focuses on the fundamental principles involved in the analysis and design of control systems. At its core, controlling a system means influencing its behavior to achieve a specific objective, as described by Sontag \cite{sontag2013mathematical}.
	
	Within control theory, systems can generally be classified into two types: deterministic and stochastic. While deterministic models offer a clear representation of many naturally occurring phenomena, they do not account for all aspects of real-world systems. As Kendrick \cite{kendrick2005stochastic} explains, deterministic control problems assume no uncertainty. In contrast, stochastic models incorporate elements of randomness, making them better suited for capturing the inherent unpredictability of real-life processes. Fleming and Rishel \cite{MR454768} distinguish between these two approaches in their work, devoting separate discussions to deterministic control theory and stochastic control of Markov diffusion processes. As Tiberio \cite{tiberio2004stochastic} points out, stochastic models introduce random variability into systems by including an error term in the differential equations that describe them. Therefore, the authors are enthusiastic about studying the controllability of stochastic control systems.

	Like deterministic systems, control theory for stochastic systems can be divided into two parts. The first part is control theory for stochastic finite dimensional systems governed by stochastic ordinary differential equations, and the second part is that for stochastic distributed parameter systems, described by stochastic differential equations in infinite dimensions, typically by stochastic partial differential equations.   One can find a huge list of publications on control theory for stochastic finite dimensional systems and its applications, say, in mathematical finance. In our opinion, control theory for stochastic distributed parameter systems is still at its very beginning stage.
	
	The controllability theory of stochastic distributed control systems mainly consists of three crucial notions: exact, approximate, and null controllability. In real-world systems, achieving exact controllability is often impossible due to noise, disturbances, or limitations in the control system itself. Therefore, approximate controllability problems are one of the most fundamental issues in the field of control engineering. Approximate controllability is also relevant for stochastic differential equations, where the system's behavior is influenced by random noise. For the controllability results described by stochastic evolution equations, we mention \cite{tiberio2004stochastic, MR4363403}.
	
	Our efforts here are prompted by the growing literature on the mathematical models for mechanical problems with nonconvex and nonsmooth energy functions. The lack of convexity and differentiability in boundary value problems in mechanics and engineering leads to new types of variational expressions called hemivariational inequalities. Many problems from nonsmooth contact mechanics involving multivalued and non monotone constitutive laws with boundary conditions can be modeled by means of hemivariational inequalities. The problem of nonmonotone semipermeability membranes and temperature control of a medium through the boundary regulated by the temperature in the interior or at the boundary leads to the evolution of hemivariational inequalities. We mention the papers \cite{MR647588}.
	
	Due to the importance of stochastic evolution equations and hemivariational inequalities in theoretical and real-life applications, finding its existence, controllability results, and other quantitative and qualitative properties in infinite dimensional spaces is significant. Therefore, in this paper, we consider the control problem governed by stochastic hemivariational inequality:
	\begin{equation}\label{HV1}
		\begin{cases}
			\langle -q^{\prime}(t)+Aq(t)+Bu(t)+\sigma(t)\frac{dW(t)}{dt}, v^*\rangle_{X} +F^0(t, q(t); v^*)\ge 0, \forall v^*\in X, ~~t\in I=[0,a]\\
			q(0)=x_0.
		\end{cases}
	\end{equation}
	Here, $\langle \cdot, \cdot \rangle_{X}$ denotes the inner product of the Hilbert space $X$. The closed linear operator $A:D(A)\subset X\to X$ generates a strongly continuous semigroup $\{T(t)\}_{t\ge 0}$ of bounded linear operators acting on $X$. The control operator $B: U\to X$ is bounded linear, acting on a separable Hilbert space $U$. The map $\sigma(t)\in \Sigma(t, q(t))$ for a.a. $t\in I$, where $\Sigma: I\times X\multimap \mathcal{L}^2_0$ is a nonempty, bounded, closed, convex valued multimap.
	The notation $F^0(t, \cdot;\cdot)$ stands for the generalized Clarke directional derivative (see \cite[page 25]{MR1058436}) of a locally Lipschitz function $F(t,\cdot): X\to \mathbb{R}$;  $\{W(t):t\ge 0\}$ is a given $\mathcal{K}$- valued Brownian motion or Wiener process with a finite trace nuclear covariance operator $Q\ge 0$; here $\mathcal{K}$ is Hilbert space with inner product $\langle \cdot, \cdot\rangle_{\mathcal{K}}$ and the norm $\norm{\cdot}_{\mathcal{K}}$ and $\mathcal{L}^2_0$ will be given later.

	We study the approximate controllability of the problem \eqref{HV1}, that means for any preassigned $\epsilon>0$, for any initial state $x_0\in X$ and any final state $z\in L^2_{\mathcal{F}_a}(\Omega,X)$ there exists a control $u_{\epsilon}\in L^2_{\mathbb{F}}(I,U)$ such that the corresponding solution $q_{\epsilon}\in C_{\mathbb{F}}(I, L^2(\Omega,X))$ satisfies
	\begin{equation}
		\mathbb{\mathbb{E}}\norm{q_{\epsilon}(a)-z}^2<\epsilon.
	\end{equation}
	The choice of \( z \in L^2_{\mathcal{F}_a}(\Omega, X) \) is justified for two key reasons. Firstly, the space \( L^2_{\mathcal{F}_a}(\Omega, X) \) serves as the natural state space for the variable \( q \). Secondly, in certain applications, it becomes necessary to consider cases where \( x_1 \in L^2_{\mathcal{F}_a}(\Omega, X) \). For instance, in deriving the Pontryagin maximum principle for infinite-dimensional control systems with terminal constraints, the concept of a finite co-dimensional condition is introduced (see \cite{li2012optimal}). This condition implies that the attainable set at time \( a \) is sufficiently large within the state space. Extending such results to stochastic settings requires examining the relationship between the reachable set and the space \( L^2_{\mathcal{F}_a}(\Omega, X) \).
	
	\vspace{0.2cm}
	
	Many real-world phenomena—such as population dynamics, stock market fluctuations, weather forecasting models, and heat conduction in materials with memory—are influenced by inherent randomness \cite{zhang2021analysis,fischer2019electric,omar2021fractional,guido2020feasibility}. Due to the presence of noise, deterministic models often undergo significant modifications. In fact, stochastic disturbances are not only unavoidable but also widespread in both natural and engineered systems. Consequently, deterministic models frequently fail to capture the true nature of fluctuations observed in practice. To better reflect reality, it becomes essential to incorporate stochastic processes into the modeling framework. In numerous cases, either the initial data and system parameters are randomly perturbed, or the underlying dynamics are inherently stochastic. This necessitates the integration of stochastic effects into the study of differential systems, leading naturally to the formulation and analysis of stochastic evolution equations.

	From our perspective, there are significant challenges associated with analyzing controllability problems in stochastic control systems. As noted in \cite{MR4363403, MR3932620, MR1402663}, the formulation of stochastic problems and the analytical techniques employed can differ markedly from those used in the deterministic setting. It is well known that if the deterministic control problem \eqref{linear deterc} is exactly controllable at time $a$ using controls from the space $L^1$ in time, then exact controllability can also be achieved using time-analytic controls. In fact, this result remains valid if the control space $L^1(I, \mathbb{R}^m)$ is replaced by $L^p(I, \mathbb{R}^m)$ for any \( p \in [1, \infty] \).
	
	In contrast, the situation is significantly different even in the most basic stochastic cases. For instance, consider the system
	\begin{equation}\label{51}
		dx(t) = [b x(t) + u(t)]\,dt + \sigma\,dW(t),
	\end{equation}
	where \( b \) and \( \sigma \) are given constants. It is evident that this system is exactly controllable when the controls \( u(\cdot) \) belong to the space \( L^1_{\mathbb{F}}(\Omega, L^2([0, T], \mathbb{R})) \) (see \cite{MR2984588}). However, exact controllability fails if the control space is restricted to \( L^2_{\mathbb{F}}(\Omega, L^2([0, T], \mathbb{R})) \).

	Similar to the deterministic case, exact controllability in stochastic control problems faces several limitations and seldom holds. Consider the following simple one-dimensional controlled stochastic differential equation:
	\begin{equation}
		\begin{cases}
			dy(t) = u(t)\,dt, \\
			y(0) = y_0,
		\end{cases}
	\end{equation}
	It has been shown in \cite{MR1402663} that this system is not exactly controllable when the admissible controls \( u(\cdot) \) are taken from the space \( L^2_{\mathbb{F}}(I, L^p(\Omega)) \), for any \( p \in (1, \infty] \). Consequently, it becomes necessary to consider a weaker notion of controllability, namely, approximate controllability. However, the authors in \cite{MR3932620} also demonstrated that the system generally fails to be approximately controllable. Interestingly, they further proved that if the control operator \( B \) is invertible, then the stochastic control system is approximately controllable.

	The study of the controllability of evolution inclusion problems governed by hemivariational inequalities involving Clarke generalized directional derivatives in practical engineering applications will be more important to the system's nonsmooth analysis and optimization. Therefore, many researchers have paid increasing attention to the solvability and control problems of stochastic evolution inclusions of Clarke subdifferential types\cite{MR3499943, MR3407227}.

	The study of approximate controllability of deterministic hemivariational inequalities in Hilbert spaces is well established. The authors \cite{MR4722465, MR4173835, MR3407227} have largely contributed to this field. Recently, the authors in \cite{MR3499943} studied the approximate controllability of stochastic hemivariational inequality problem \eqref{HV1} for the single-valued map $\Sigma$. Also, in \cite{MR3407227}, the authors studied the controllability for stochastic evolution inclusions involving Clarke subdifferentials. In \cite{MR3629137}, the authors discussed the existence of mild solutions for the stochastic evolution inclusion involving Clarke subdifferential and multimap $\Sigma$. In the present work, we generalize the results obtained in \cite{MR3499943} by incorporating multimap $\Sigma$. 
	We hope that the results of this paper pave the way to a better understanding of controllability problems for infinite dimensional stochastic control systems.

	The paper is organized as follows:
	\begin{itemize}
		\item[(1)] Section 1 is the Introduction, where we provide the problem statement and related literature review.
		\item[(2)] In Section 2, we provide the preliminary results used to prove the main results.
		\item[(3)] Section 3 is based on the existence of solutions to the hemivariational control problem \eqref{HV1}.
		\item[(4)] In Section 4, we study the approximate controllability of the hemivariational control problem \eqref{HV1}.
		\item[(5)] In Section 5, we justify the abstract findings of this paper through an example of a stochastic heat propagation problem.
		\item[(6)] Section 6 is used as an Appendix Section. 
	\end{itemize}

	\section{Preliminaries}
	In this section, we provide all the necessary preliminaries to make it a well-contained article. These preliminaries contain some definitions and results about stochastic integral in Banach spaces, multivalued maps, nonsmooth analysis, and control theory.
	\subsection{Some Basic Facts on Stochastic Calculus} 
	In this section, we recall definitions and facts about stochastic processes and stochastic calculus in Hilbert spaces. In this regard, readers are encouraged to consult the thesis \cite{MR3860612} and the papers \cite{MR3634281, MR1784435}.

	Throughout this section, we assume that \( (\Omega, \mathcal{F}, P) \) be a complete probability space endowed with a normal filtration \( \{\mathcal{F}_t\}_{t \in [0, a]} \) satisfying the usual conditions—namely, right-continuity and completeness (i.e., \( \mathcal{F}_0 \) contains all \( P \)-null sets). 
	
	Let $X$ be a separable Banach space with a norm $\norm{\cdot}_X$.
	\begin{definition}
		Let $I\subset \mathbb{R}$. A family $\{\xi(t)\}_{t\in I}$ of $X$-valued random variables $\xi(t),~t\in I$, defined on $\Omega$ is called an $X$-valued stochastic process on $I$.  
	\end{definition}
	\begin{definition}
		For each $\omega\in \Omega$, a mapping defined by $[0,a]\ni t\mapsto \xi(t,\omega)\in X$ is called trajectory of the process $\xi$.
	\end{definition}
	\begin{definition}
		The process $\xi$ is called continuous if the trajectories of $\xi$ are $P$- a.s. continuous on $I$, that means there exists $\bar{\Omega}\in \mathcal{F}$ with $P(\bar{\Omega})=1$ such that for each $\omega\in \bar{\Omega}$, the mapping
		\begin{equation}
			[0, a]\ni t\mapsto \zeta(t,\omega)\in X
		\end{equation}
		is continuous on $[0, a]$.
	\end{definition}
	\begin{definition}
		The process $\xi$ is called measurable if the mapping $\xi:[0,a]\times \Omega\to X$ is $\mathcal{B}([0,a])\otimes \mathcal{F}$- measurable.
	\end{definition}
	\begin{definition}
		A family $\{\mathcal{F}_t\}_{t\ge 0}$ of $\sigma$- fields $\mathcal{F}_t\subset \mathcal{F}$ is called a filtration if for any $0\le s\le t<\infty$, $\mathcal{F}_s\subset \mathcal{F}_t$. 
	\end{definition}
	Throughout this article, we assume that $\mathbb{F}=\{\mathcal{F}_t\}_{t\ge 0}$ is a filtration.
	\begin{definition}
		The process $\xi$ is said to be adapted to the filtration $\mathbb{F}$ if for each $t\in I$, $\xi(t)$ is $\mathcal{F}_t$-measurable.
	\end{definition}
	\begin{definition}
		The process $\xi$ is called progressively measurable if for each $t\in I$, the following mapping
		\begin{equation}
			[0,t]\times \Omega\ni (s,\omega)\mapsto \xi(s,\omega)\in X
		\end{equation}
		is $\mathcal{B}([0,t])\otimes \mathcal{F}_t$-measurable.
	\end{definition}
	\begin{proposition}
		If the process $\xi$ is measurable and adapted to $\mathbb{F}$, then it has an $\mathbb{F}$- progessively measurable modification $\zeta$, that means
		\begin{equation}
			P(\{\omega\in \Omega: \xi(t,\omega)\neq \zeta(t,\omega) \})=0,~\text{for every}~t\in I.
		\end{equation}
	\end{proposition}
	\begin{proposition}
		If the process $\xi$ is continuous and adapted to $\mathbb{F}$, then it has an $\mathbb{F}$- progressively measurable modification.
	\end{proposition}
	
	We consider a \( Q \)-Wiener process defined on \( (\Omega, \mathcal{F}, P) \), where \( Q \) is a linear, bounded, and self-adjoint covariance operator on \( \mathcal{K} \), satisfying \( \operatorname{tr}(Q) < \infty \). Suppose there exists a complete orthonormal basis \( \{e_n\}_{n \in \mathbb{N}} \subset \mathcal{K} \), along with a bounded sequence of nonnegative real numbers \( \{\lambda_n\}_{n \in \mathbb{N}} \), such that \( Qe_n = \lambda_n e_n \) for all \( n \in \mathbb{N} \). Further, let \( \{W_n\}_{n \in \mathbb{N}} \) be a sequence of independent standard Brownian motions. Then, for every \( e \in \mathcal{K} \) and \( t \in I \), the \( Q \)-Wiener process \( W(t) \) satisfies:
	\begin{equation}
		\langle W(t), e \rangle_{\mathcal{K}} = \sum_{n=1}^{\infty} \sqrt{\lambda_n} \langle e_n, e \rangle_{\mathcal{K}} W_n(t),
	\end{equation}
	and the filtration \( \mathcal{F}_t \) coincides with the natural filtration \( \mathcal{F}_t^w \) generated by the Wiener process \( W \), that is, \( \mathcal{F}_t = \mathcal{F}_t^w = \sigma\{W(s): 0 \leq s \leq t\} \).
	
	Let \( \mathcal{L}^2_0 = L^2(Q^{1/2} \mathcal{K}; H) \) denote the space of all Hilbert-Schmidt operators from the image of \( Q^{1/2} \) to \( H \), equipped with the inner product
	\[
	\langle \Psi, \Gamma \rangle_{\mathcal{L}^2_0} = \operatorname{tr}(\Psi Q \Gamma^*).
	\]
	
	Moreover, let \( L^2(\Omega, H) \) denote the Banach space of all \( \mathcal{F}_a \)-measurable, square-integrable \( H \)-valued random variables, endowed with the norm
	\[
	\|f\| = \left( \mathbb{E} \|f(\cdot, \omega)\|_H^2 \right)^{1/2}.
	\]

	We now define the space
	\begin{equation}
		L^2_{\mathbb{F}}(\Omega, C(I,X))=\{\phi:[0,a]\times \Omega\to X: \phi~\text{is continuous}~\mathbb{F}~\text{adapted and}~\mathbb{E}\left(\sup_{t\in [0,a]}\norm{\phi(t)}_X^2\right)<\infty\}.
	\end{equation}
	It is noted that the space $ L^2_{\mathbb{F}}(\Omega, C(I,X))$ is a Banach space under the norm
	\begin{equation}
		\norm{\phi}_{ L^2_{\mathbb{F}}(\Omega, C(I,X))}=\mathbb{E}\left(\sup_{t\in [0,a]}\norm{\phi(t)}_X^2\right)^{\frac{1}{2}}.
	\end{equation}
	Also, define the space
	\begin{equation}
		C_{\mathbb{F}}(I,L^2(\Omega, X))=\{\phi:[0,a]\times \Omega\to X: \phi~\text{is}~\mathbb{F}~\text{adapted and}~\phi(\cdot):I\to L^2(\Omega, X)~\text{is continuous}\}.
	\end{equation}
	It is noted that the space $C_{\mathbb{F}}(I,L^2(\Omega, X))$ is a Banach space under the norm
	\begin{equation}
		\norm{\phi}_{C_{\mathbb{F}}(I,L^2(\Omega, X))}=\sup_{t\in [0,a]}\mathbb{E}\left(\norm{\phi(t)}_X^2\right)^{\frac{1}{p}}.
	\end{equation}
	It is clear that
	\begin{equation}
		\norm{\phi}_{C_{\mathbb{F}}(I,L^2(\Omega, X))}\le \norm{\phi}_{L^p_{\mathbb{F}}(\Omega, C(I,X))},
	\end{equation}
	for every $\phi:[0,a]\times \Omega\to X: \phi~\text{is continuous}~\mathbb{F}~\text{adapted}$.
	$L_{\mathbb{F}}^2(I,X)$ will denote the Banach space of all $\mathcal{F}_t$- adapted measurable random processes defined on $I$ with values in $X$ and the norm
	\begin{equation}
		\norm{f}_{L_{\mathbb{F}}^2(I,X)}=\left(\mathbb{E}\int_{0}^{a}\norm{f(t)}_X^2dt\right)^{\frac{1}{2}}.
	\end{equation}
	Similarly, we define the Hilbert space
	$L_{\mathbb{F}}^2(I,U)$ of all $\mathcal{F}_t$- adapted measurable random processes defined on $I$ with values in $U$ and the norm
	\begin{equation}
		\norm{u}_{L_{\mathbb{F}}^2(I,U)}=\left(\mathbb{E}\int_{0}^{a}\norm{u(t)}_U^2dt\right)^{\frac{1}{2}}.
	\end{equation}

	We now state some properties of stochastic integration in Hilbert spaces.
	
	\begin{theorem}\label{TS1}\cite{MR1784435}
		Let $(\Omega, \mathcal{F}, P)$. where $\mathbb{F}=(\mathcal{F}_t)_{t\ge 0}$ be a filtered probability space. Assume $(X, \norm{\cdot}_X)$ is a Hilbert space and $\mathcal{K}$ is a separable Hilbert space endowed with an inner product $\langle \cdot, \cdot\rangle_{\mathcal{K}}$. Assume that  $\{W(t):t\ge 0\}$ is a given $\mathcal{K}$- valued Brownian motion or Wiener process with a finite trace nuclear covariance operator $Q\ge 0$. If $\sigma$ is a $\mathcal{L}^2_0$-valued $\mathbb{F}$-progressively measurable process on $[0, a], ~a>0$ such that 
		\begin{equation}
			\mathbb{E}\left[\int_{0}^{a}\norm{\sigma(t)}^2_{\mathcal{L}^2_0}dt\right]<\infty,
		\end{equation}
		then the stochastic integral $\int_{0}^{a}\sigma(t)dW(t)$ is well defined and there exists $C>0$ such that 
		\begin{equation}
			\mathbb{E}\left[\norm{\int_{0}^{a}\sigma(t)dW(t)}_X^2\right]\le C 	\int_{0}^{a}\mathbb{E}\norm{\sigma(t)}^2_{\mathcal{L}^2_0}dt,
		\end{equation}
		that is, $\displaystyle \int_{0}^{a}\sigma(t)dW(t)$ belongs to the space $L^2(\Omega, \mathcal{F},P,X)$.
	\end{theorem}
	\begin{theorem}\cite{MR1784435}
		Suppose that all the assumptions of Theorem \ref{TS1} are satisfied. Let $\sigma$ be a $\mathcal{L}^2_0$-valued $\mathbb{F}$-progressively measurable process on $[0, a], ~a>0$ such that 
		\begin{equation}
			\mathbb{E}\left[\int_{0}^{a}\norm{\sigma(t)}^2_{\mathcal{L}^2_0}dt\right]<\infty,
		\end{equation}
		Then the process $\{\phi(t)\}_{t\in [0, a]}$ defined by
		\begin{equation}
			\phi(t)=\int_{0}^{t}\sigma(s)dW(s), ~t\in [0, a]
		\end{equation}
		is a martingale and has a continuous modification; in particular, it is $\mathbb{F}$- progressively measurable. Moreover, there exists a constant $c>0$ (independent of $\sigma$) such that 
		\begin{equation}
			\mathbb{E}\left[\sup_{t\in [0, a]}\norm{\phi(t)}_X^2\right]\le c\int_{0}^{t}	\mathbb{E}\norm{\sigma(s)}^2_{\mathcal{L}^2_0}ds.
		\end{equation}
	\end{theorem}
	\begin{theorem}\label{stochastic convolution estimate}\cite{MR3953459}
		Let $(\Omega, \mathcal{F},\mathbb{F}, P)$, where $\mathbb{F}=(\mathcal{F}_t)_{t\ge 0}$ be a filtered probability space. Assume $(X, \norm{\cdot}_X)$ is a Hilbert space and $\mathcal{K}$ is a separable Hilbert space endowed with an inner product $\langle \cdot, \cdot\rangle_{\mathcal{K}}$. Assume that  $\{W(t):t\ge 0\}$ is a given $\mathcal{K}$- valued Brownian motion or Wiener process with a finite trace nuclear covariance operator $Q\ge 0$. Let $\{S(t)\}_{t\ge 0}$ be a contraction $C_0$- semigroup on $X$. If $\sigma$ is a $\mathcal{L}^2_0$-valued $\mathbb{F}$-progressively measurable process on $[0, a], ~a>0$ such that 
		\begin{equation}
			\mathbb{E}\left[\int_{0}^{a}\norm{\sigma(t)}^2_{\mathcal{L}^2_0}dt\right]<\infty.
		\end{equation}
		Then there exists a constant $K_a>0$ such that
		\begin{equation}
			\mathbb{E}\left[\sup_{t\in [0, a]}\norm{\int_{0}^{t}S(t-r)\sigma(r)dW(r)}_X^2\right]\le K_a 	\int_{0}^{a}\mathbb{E}
			\norm{\sigma(t)}^2_{\mathcal{L}^2_0}dt.
		\end{equation}
	\end{theorem}
	\subsection{Multivalued Maps}
	In this section, we present important facts from multivalued analysis that are used in this article. For this, we refer to the books \cite{MR1831201},\cite{MR2976197}.
	\begin{definition}
		A multivalued map $\Gamma:X\multimap Y$ is said to be
		\begin{itemize}
			\item  convex  valued if the set $\Gamma(x)$ is convex for every $x\in X$ and closed valued if $\Gamma(x)$ is closed for all $x\in X$.
			\item  bounded if it maps bounded sets in $X$ into bounded sets in $Y$. Namely, if $B_X$ is a bounded set in $X$, then the set $\Gamma(B_X)$ is bounded in $Y$, that is, 
			\begin{equation*}
				\sup_{x\in B_X}\{\sup\{\norm{y}:y\in \Gamma(x)\}\}<\infty.
			\end{equation*}
		\end{itemize}
	\end{definition}
	We now define the continuity property of the multifunction. A multivalued map $\Gamma: X\multimap Y$ has 
	\begin{itemize}
		\item  upper semicontinuity property at the point $x_0\in X$ if for each open set $O_Y\subset Y$ containing $\Gamma(x_0)$, there exists an open neighbourhood $O_X$ of $x_0$ such that $\Gamma(O_X)\subset O_Y$. If $\Gamma$ is upper semicontinuous at every $x\in X$, then $\Gamma$ is upper semicontinuous.
		\item  strongly weakly closed graph property if for every sequences $\{x_n\}_{n\in \mathbb{N}}\subset X$, $\{y_n\}_{n\in \mathbb{N}}\subset Y$ such that $y_n\in \Gamma(x_n)$ for all $n\in \mathbb{N}$ with $x_n\to X$ in $X$, $y_n\rightharpoonup y$ in $Y$, then we have $y\in \Gamma(x)$.
	\end{itemize}
	We now move forward towards the multivalued measurable functions.\\
	Let $I\subset \mathbb{R}$ be a measurable set and $\mathcal{L}$ be the $\sigma$- algebra of subsets of $I$. Also let $X$  be a separable reflexive Banach space.
	\begin{definition}
		A multivalued map $\Gamma:I\multimap X$ is said to be measurable if for every $C_X\subset X$ closed, the set 
		\begin{equation*}
			\{t\in I: \Gamma(t)\cap C_X\neq \phi\}\in \mathcal{L}.
		\end{equation*}
	\end{definition}
	Let $\mathcal{B}(X)$ be the Borel $\sigma$- algebra of subsets of $X$.
	\begin{definition}
		We say that the multivalued mapping $\Gamma: I\times X\multimap X$ is $\mathcal{L}\times \mathcal{B}(X)$ measurable if
		\begin{equation*}
			\Gamma^{-1}(C_X)=\{(t,x)\in I\times X: \Gamma (t,x)\cap C_X\neq \phi\}\in \mathcal{L}\times \mathcal{B}(X),
		\end{equation*}
		for any closed set $C_X\subset X$.
	\end{definition} 
	The following fixed point theorem is useful in this article.
	\begin{theorem}\label{fixed}\cite{MR46638}
		Let $X$ be a Hausdorff locally convex topological vector space, $K$ a compact convex subset of $X$, and $\Gamma: K\multimap K$ an upper semicontinuous multimap with closed, convex values. Then the multimap $\Gamma$ has a fixed point in $K$.
	\end{theorem}
	\subsection{Nonsmooth Analysis}
	In what follows, let us define Clarke subdifferential (see \cite[page 27]{MR1058436}) of a locally Lipschitzian functional $F: X\to \mathbb{R}$. We denote by $F^0(y;z)$ the Clarke generalized directional derivative of $F$ at $y$ in the direction $z$, that is
	\begin{equation}
		F^0(y;z)=\lim_{\epsilon\to 0^+}\sup_{\xi\to y}\frac{F(\xi+\epsilon z)-F(\xi)}{\epsilon}.
	\end{equation} 
	Recall also that the Clarke subdifferential of $F$ at $y$, denoted by $\partial F(y)$ is a subset of $X$ given by
	\begin{equation}
		\partial F(y)=\{y^*\in X: F^0(y;z)\ge \langle y^*,z\rangle_{X}, \forall z\in X\}.
	\end{equation}
	The following basic properties of the generalized directional derivative and the generalized gradient are important in our main results.
	\begin{lemma}\label{Lemma 2.14} \cite[Proposition 2.1.2]{MR1058436}
		If $F: X\to \mathbb{R}$ is a locally Lipschitz function, then
		\begin{itemize}
			\item[(i)] for every $z\in X$, one has 
			\begin{equation}
				F^0(y;z)=\max\{\langle y^*,z\rangle_{X}: ~\text{for all}~y^*\in \partial F(y)\}.
			\end{equation}
			\item[(ii)] for every $y\in X$, the gradient $\partial F(y)$ is a nonempty, convex, weak$^*$-compact subset of $X$ and $\norm{y^*}_{X}\le K$ for any $y^*\in \partial F(y)$ (where $K>0$ is the Lipschitz constant of $F$ near $y$).
			\item[(iii)] the graph of the generalized gradient $\partial F$ is closed in $X\times w-X$ topology, that is, if $\{y_n\}_{n\in \mathbb{N}}\subset X$ and $\{y_n^*\}_{n\in \mathbb{N}}\subset X$ are sequences such that $y_n^*\in \partial F(y_n)$ and $y_n\to y$ in $X$, $y_n^*\to y^*$ weakly in $X$, then $y^*\in \partial F(y)$. Here, $w-X$ denotes the Banach space $X$ furnished with the weak topology.
			\item[(iv)] the multifunction $y\multimap \partial F(y)\subset X$ is upper semicontinuous from $X$ into $w-X$.
		\end{itemize} 
	\end{lemma}
	\begin{lemma}\label{Lemma 2.15}  \cite[Proposition 3.44]{MR2976197}
		Let $X$ be a separable reflexive Banach space and $F: I\times X\to \mathbb{R}$ be a function such that $F(\cdot,x)$ is measurable for all $x\in X$ and $F(t,\cdot)$ is locally Lipschitz on $X$ for almost all $t\in I$. Then, the multifunction $(t,x)\multimap \partial F(t,x)$ is measurable, where $\partial F$ denotes the Clarke subdifferential of $F(t,\cdot)$.
	\end{lemma}
	\subsection{Some basic results about controllability}
	In this section, we discuss some basic results of deterministic and stochastic linear control systems. For this section, we refer the readers to the book \cite{MR4363403}.

	We consider the linear stochastic control system associated with Problem \eqref{HV2}.
	\begin{equation}\label{linear stochastic}
		\begin{cases}
			q^{\prime}(t)=Aq(t)+Bu(t)+\sigma(t)dW(t),~t\in I\\
			q(0)=x_0.
		\end{cases}
	\end{equation}
	In the above $u\in L^2_{\mathbb{F}}(I,U)$ and $\sigma\in L^2_{\mathbb{F}}(I,\mathcal{L}^2_0)$.
	
	We begin with the notion of exact controllability of \eqref{linear stochastic}.
	\begin{definition}
		The system \eqref{linear stochastic} is called exactly controllable at time $a$ if for any $x_0\in X$ and $x_1\in L^2_{\mathcal{F}_a}(\Omega,X)$, there is a control $u\in L^2_{\mathbb{F}}(I,U)$ such that the corresponding solution $q(\cdot)$ to \eqref{linear stochastic} fulfils that $q(a)=x_1$ a.s.
	\end{definition}
	
	The authors in Theorem 3.1. concludes that the stochastic linear control system \eqref{linear stochastic} is not exactly controllable when the control, which is $L^2$ with respect to the time variable, is only acted in the drift term.
	
	Now, we introduce null and approximate controllability. 
	\begin{definition}
		The system \eqref{linear stochastic} is called null controllable at time $a$ if for any $x_0\in X$ there is a control $u\in L^2_{\mathbb{F}}(I, U)$ such that the corresponding solution $q(\cdot)$ to \eqref{linear stochastic} fulfills that $q(a)=0$ a.s.
	\end{definition}
	\begin{definition}
		The system \eqref{linear stochastic} is called approximately controllable at time $a$ if for any $x_0\in X$ and $x_1\in L^2_{\mathcal{F}_a}(\Omega,X)$ and $\epsilon>0$, there is a control $u\in L^2_{\mathbb{F}}(I,U)$ such that the corresponding solution $q(\cdot)$ to \eqref{linear stochastic} fulfils that 
		\begin{equation}
			\norm{q(a)-x_1}_{L^2_{\mathbb{F}_a}(\Omega,X)}<\epsilon
		\end{equation}
	\end{definition}
	We now consider the deterministic control system
	\begin{equation}\label{linear deterc}
		\begin{cases}
			q^{\prime}(t)=Aq(t)+Bu(t),~t\in I\\
			q(0)=x_0.
		\end{cases}
	\end{equation}
	In the above $u\in L^2(I,U)$.

	It has been shown that the deterministic linear control system is null controllable at time \( a \) if and only if the corresponding stochastic linear control system is also null controllable. Nevertheless, exact or approximate controllability of the deterministic system \eqref{linear deterc} does not guarantee the same for the stochastic system \eqref{linear stochastic}. 
	
	Additionally, it is worth noting that if the operator \( A \) generates a \( C_0 \)-group on \( X \), then the null controllability of \eqref{linear deterc} implies its exact or approximate controllability. This implication, however, does not extend to the stochastic system \eqref{linear stochastic}, highlighting a novel phenomenon specific to the stochastic framework.

	For deterministic control systems, if we assume that the control operator $B$ is invertible, then it is easy to get controllability results. The following theorem shows that the same holds for stochastic control systems.
	\begin{theorem}\label{stochastic approximate controllability}\cite{MR4363403}
		If $B\in \mathcal{L}(U, X)$ is invertible, then the system \eqref{linear stochastic} is both null and approximately controllable at any time $a>0$.
	\end{theorem}
	\begin{corollary}
		If $B\in \mathcal{L}(U, X)$ is invertible, then the linear control system \eqref{linear deterc} is both null and approximately controllable at any time $a>0$.
	\end{corollary}

	Let us define $G(a):X\to X$ as follows:
	\begin{equation}\label{Controllability gramian}
		G(a)y^*=\int_{0}^{a}T(a-s)BB^*T^*(a-s)y^*ds,~ y^*\in X.
	\end{equation} 
	
	We now give a criterion for the approximate controllability of the system \eqref{linear deterc} in terms of the adjoint of the semigroup.
	\begin{theorem}\label{Thm4.2}
		The system \eqref{linear deterc} is approximately controllable on $I$ if and only if $	\text{for every}~y^*\in X$,
		\begin{equation*}
			B^*T^*(a-t)y^*=0, ~\text{for all }~t\in I \Rightarrow y^*=0.
		\end{equation*}
	\end{theorem}
	The following simple observation is crucial, and it is immediate from the positivity of the operator $G(a)$.
	\begin{lemma}\label{Lem4.5}\cite[Remark 4.2]{MR4104454}
		Suppose the system \eqref{linear deterc} is approximately controllable on $I$; then $G(a)$ (defined in \ref{Controllability gramian}) is injective.
	\end{lemma}
	The following lemmas are important and a direct consequence of \cite[Lemma 2.2]{MR2046377} and \cite[Lemma 4.4]{MR4104454}
	\begin{lemma}\label{Lem3.1}
		Suppose that $X$ is a separable reflexive Banach space. Then the map $\epsilon I+G(a)$ is invertible for every $\epsilon>0$, and satisfies the following estimates
		\begin{equation*}
			\norm{\epsilon(\epsilon I+G(a))^{-1}y}\le\norm{y},
		\end{equation*}
		for all $y\in X$. 
	\end{lemma}
	
	\begin{lemma}\label{Thm6.0}
		Let $U$ be a separable Hilbert space and $X$ be a Banach space with dual $X$. Assume that the linear control system \eqref{linear deterc}
		is approximately controllable on $I$. Additionally assume that there exists a relatively compact set $K\subset X$, and $x_{\epsilon} \in X$ such that
		\begin{equation}
			x_{\epsilon}\in \epsilon(\epsilon I+G(a))^{-1}(K), ~\epsilon>0,
		\end{equation}
		where $G(a)$ is defined in \eqref{Controllability gramian}.
		Then there is a sequence $\epsilon_n\to 0$ as $n\to \infty$ such that $x_{\epsilon_n}\to 0$ as $n\to \infty$.
	\end{lemma}
	\section{Concepts of Solutions}
	In this section, we discuss the solution concepts for the hemivariational control problem \eqref{HV1}.

	First of all, let us define the weak solution to the hemivariational control problem \eqref{HV1}.
	\begin{definition}\label{weak hemivariational}
		For fixed $u(\cdot)\in L^2_{\mathbb{F}}(I,U)$, an $X$-valued, $\mathbb{F}$- adapted, continuous stochastic process $q(\cdot)$ is called a weak solution to \eqref{HV1} if
		\begin{itemize}
			\item[(1)] $Bu(\cdot)\in L^1([0,a], X)$ $P$- a.s.and there exists $\sigma(t)\in \Sigma(t,q(t))$ for a.a. $t\in I$ which satisfies
			\begin{equation}
				\mathbb{E}\left[\int_{0}^{a}\norm{\sigma(t)}^2_{\mathcal{L}^2_0}dt\right]<\infty.
			\end{equation}
			\item[(2)] For all $t\in [0,a]$, for every $\xi\in D(A^*)$ we have
			\begin{align*}
				-\langle q(t), \xi\rangle_{X}+&\langle x_0, \xi \rangle_{X}+\int_{0}^{t}\langle q(s), A^*\xi\rangle_{X}ds+\int_{0}^{t}\langle Bu(s), \xi\rangle_{X}ds\\
				+& \int_{0}^{t}\langle \langle \sigma(s), \xi\rangle \rangle_{\mathcal{L}^2_0,X}dW(s)+\int_{0}^{t}F^0(s,q(s); \xi)ds\ge 0,~P ~\text{a.s}..
			\end{align*}
			Here, for $f\in \mathcal{L}^2_0$ and $\xi\in X$ the mapping $\langle \langle f, \xi \rangle \rangle_{\mathcal{L}^2_0,X}:\mathcal{K}\to \mathbb{R}$ is defined by $\langle \langle f, \xi \rangle \rangle_{\mathcal{L}^2_0,X}(v)=\langle f(v), \xi\rangle_{X}$.		
		\end{itemize}
	\end{definition}
	We now reformulate the hemivariational control problem \eqref{HV1} to the following differential inclusion involving Clarke subdifferential:
	\begin{equation}\label{HV2}
		\begin{cases}
			q^{\prime}(t)\in Aq(t)+Bu(t)+\partial F(t,q(t))+\Sigma(t,q(t)))\frac{dW(t)}{dt}  ~~t\in I=[0,a]\\
			q(0)=x_0.
		\end{cases}
	\end{equation}
	In the above, the multimap $\partial F:I\times X\multimap X$ stands for the Clarke subdifferential (see \cite[page 27]{MR1058436}) of a locally Lipschitz function $F(t,\cdot): X\to \mathbb{R}$.
	
	We prove that the solutions of the differential inclusion problem \eqref{HV2} and hemivariational inequality problem \eqref{HV1} are equivalent. Before doing so, we need to introduce the solution concept for the inclusion problem \eqref{HV2}.
	
	In order to define the solution concept, we introduce the measurable selection multimaps $S_{\partial F}: C_{\mathbb{F}}(I, L^2(\Omega, X))\multimap L^2_{\mathbb{F}}(I, X)$ as follows:
	\begin{equation}\label{SF}
		S_{\partial F}(q)=\{f\in  L^2_{\mathbb{F}}(I,X): f(t)\in \partial F(t,q(t))~\text{for a.a.}~t\in I\},~q\in C_{\mathbb{F}}(I,L^2(\Omega,X)),
	\end{equation}
	and  $S_{\Sigma}: C_{\mathbb{F}}(I, L^2(\Omega, X))\multimap L^2_{\mathbb{F}}(I,\mathcal{L}^2_0)$ as follows:
	\begin{equation}\label{Ssigma}
		S_{\Sigma}(q)=\{\sigma\in  L^2_{\mathbb{F}}(I,\mathcal{L}^2_0): \sigma(t)\in \Sigma(t,q(t))~\text{for a.a.}~t\in I\},~q\in C_{\mathbb{F}}(I,L^2(\Omega,X)).
	\end{equation}
	Assuming that the multimaps $S_{\partial F}$ and $S_{\Sigma}$ are well defined, we define the solution concept for the inclusion problem \eqref{HV2}.

	\begin{definition}\label{weak differee}
		For fixed $u(\cdot)\in L^2_{\mathbb{F}}(I,U)$, an $X$-valued, $\mathbb{F}$- adapted, continuous stochastic process $q(\cdot)$ is called a weak solution to \eqref{HV2} if
		\begin{itemize}
			\item[(1)] $Bu(\cdot)\in L^1([0,a], X)$ $P$-a.s.; there exists $\sigma \in S_{\Sigma}$ satisfies
			\begin{equation}
				\mathbb{E}\left[\int_{0}^{a}\norm{\sigma(t)}^2_{\mathcal{L}^2_0}dt\right]<\infty.
			\end{equation}
			\item[(2)] For all $t\in [0,a]$, 
			\begin{align*}
				\langle q(t), \xi\rangle_{X}=&\langle x_0, \xi \rangle_{X}+\int_{0}^{t}\langle q(s), A^*\xi\rangle_{X}ds+\int_{0}^{t}\langle Bu(s)+ f(s), \xi\rangle_{X}ds\\
				+& \int_{0}^{t}\langle \langle \sigma(s), \xi\rangle_{\mathcal{L}^2_0,X}dW(s), P~\text{a.s.},~ \text{for all}~\xi\in D(A^*),
			\end{align*}
			for some $f\in S_{\partial F}(q)$.
		\end{itemize}
	\end{definition}
	\begin{definition}
		For fixed $u(\cdot)\in L^2_{\mathbb{F}}(I,U)$, an $X$-valued, $\mathbb{F}$- adapted, continuous stochastic process $q(\cdot)$ is called a mild solution to \eqref{HV2} if
		\begin{itemize}
			\item[(1)] $Bu(\cdot)\in L^1([0,a], X)$ $P$-a.s.; there exists $\sigma \in S_{\Sigma}$ satisfies
			\begin{equation}
				\mathbb{E}\left[\int_{0}^{a}\norm{\sigma(t)}^2_{\mathcal{L}^2_0}dt\right]<\infty.
			\end{equation}
			\item[(2)] For all $t\in [0,a]$, 
			\begin{equation}
				q(t)=T(t)x_0+\int_{0}^{t}T(t-s)[ f(s)+Bu(s)]ds+\int_{0}^{t}T(t-s)\sigma(s)dW(s), P~\text{a.s},
			\end{equation}
			for some $f\in S_{\partial F}(q)$.
		\end{itemize}
	\end{definition}
	The following theorem gives the equivalence between weak and mild solutions to the differential inclusion problem \eqref{HV2}.
	\begin{theorem}
		An $X$-valued, $\mathbb{F}$- adapted, continuous stochastic process $q(\cdot)$ is called a mild solution to \eqref{HV2} if and only if it is a weak solution to the problem \eqref{HV2}.
	\end{theorem}
	We now prove that the weak (and hence mild) solutions to the differential inclusion problem \eqref{HV2} are the weak solutions to the stochastic hemivariational control problem \eqref{HV1}. Let $q$ be a weak solution to the differential inclusion problem \eqref{HV2}. Then, from Definition \eqref{weak differee}, it follows that 
	\begin{align}\label{G}
		\langle q(t), \xi\rangle_{X}=&\langle x_0, \xi \rangle_{X}+\int_{0}^{t}\langle q(s), A^*\xi\rangle_{X}ds+\int_{0}^{t}\langle Bu(s)+ f(s), \xi\rangle_{X}ds\\
		+& \int_{0}^{t}\langle \langle \sigma(s), \xi\rangle_{\mathcal{L}^2_0,X}dW(s), ~\text{a.s}
		\notag
	\end{align}
	for some $f\in S_{\partial F}(q), \sigma\in S_{\Sigma}(q)$ and for all $\xi\in D(A^*)$. Now, from $f\in S_{\partial F}(q)$ we obtain $f(t)\in \partial F(t,q(t))$ for a.a. $t\in I$.  The definition of Clarke subdifferential permits us to write \begin{equation}
		\langle f(t), \xi\rangle_{X}\le F^0(t,q(t);  \xi) ~\text{for all}~\xi\in X.
	\end{equation}
	Integrating from $0$ to $t$ we obtain from above
	\begin{equation}
		\int_{0}^{t}\langle f(s), \xi\rangle_{X}ds\le \int_{0}^{t}F^0(s,q(s);  \xi)ds ~\text{for all}~\xi\in X.
	\end{equation}
	Hence, observing this inequality, we conclude from \eqref{G} that 
	\begin{align*}
		-\langle q(t), \xi\rangle_{X}+&\langle x_0, \xi \rangle_{X}+\int_{0}^{t}\langle q(s), A^*\xi\rangle_{X}ds+\int_{0}^{t}\langle Bu(s), \xi\rangle_{X}ds+\int_{0}^{t}F^0(s,q(s);  \xi)ds\\
		+& \int_{0}^{t}\langle \langle \sigma(s), \xi\rangle \rangle_{\mathcal{L}^2_0,X}dW(s)\ge 0, P~\text{a.s},~\text{for all}~\xi\in D(A^*).
	\end{align*}
	By means of Definition \ref{weak hemivariational}, it follows that $q$ is a weak solution to the hemivariational inequality problem \eqref{HV1}.
	Based on this discussion, we focus our attention on the differential inclusion problem \eqref{HV2} to deduce the approximate controllability result for the hemivariational control problem \eqref{HV1}.

	We assume the following hypotheses:
	\begin{itemize}
		\item[($\mathcal{T}$)] The operator $A:D(A)\subset X\to X$ generates a strongly continuous semigroup $\{T(t)\}_{t\ge 0}$ of bounded linear operators $T(t): X\to X$ and $T(t)$ is compact for $t>0$. Moreover, there exists a constant $M>0$ such that
		\begin{equation}
			\norm{T(t)}_{\mathcal{L}(X)}\le M, ~~0\le t\le a.
		\end{equation}
		
		\item[(B)] The control operator $B:U\to X$ is invertible. 
		\item[($\sigma$)] The mapping $\Sigma: I\times X\to \mathcal{L}^2_0$ has closed bounded and convex values and satisfies the following conditions:
		\begin{itemize}
			\item[($\sigma$1)] $\Sigma(\cdot,x)$ has a measurable selection for every $x\in X$, that is there exists a measurable function $\sigma: [0,a]\to \mathcal{L}^2_0$  such that $\sigma(t)\in \Sigma(t,x)$ for a.a. $t\in [0,a]$;
			\item[($\sigma$2)]$\Sigma(t,\cdot)$ is weakly sequentially closed for a.a $t\in [0,a]$, that is, it has a weakly sequentially closed graph;
			\item[($\sigma$3)] there exists $\zeta\in L^1(I,\mathbb{R}^+)$ such that for each $x\in X$
			\begin{equation}
				\norm{\Sigma(t,x)}^2_{\mathcal{L}^2_0}\le \zeta(t),~\text{for a.a.}~t\in [0,a].
			\end{equation}
		\end{itemize}
	\end{itemize} 
	
	We also assume that the function $F: I\times X\to \mathbb{R}$ satisfies the following Hypotheses.
	\begin{itemize}
		\item[(F1)] the function $t\mapsto F(t,x)$ is measurable for all $x\in X$.
		\item[(F2)] the function $x\mapsto F(t,x)$ is locally Lipschitz continuous for a.a. $t\in I$.
		\item[(F3)] there exists a function $\eta\in L^1(I, \mathbb{R}^+)$ such that
		\begin{equation}
			\norm{\partial F(t,x)}^2=\sup\{\norm{z}^2_{X}: z\in \partial F(t,x)\}\le \eta(t),~ t\in I, x\in X.
		\end{equation}
	\end{itemize}
	
	\begin{remark}
		Based on the assumption (B), we infer that the linear deterministic control system  \eqref{linear deterc} is approximately controllable in $I$. Moreover, by virtue of Theorem \ref{stochastic approximate controllability}, the stochastic control system \eqref{linear stochastic} is approximately controllable in $I$.
	\end{remark}
	\begin{remark}
		Note that under Hypotheses (F1)-(F3) by virtue of Lemmas \ref{Lemma 2.14} and \ref{Lemma 2.15} the multimap $G:I\times X\multimap X$ given by $G(t,x)=\partial F(t,x)$ satisfies the following Hypotheses (G1)-(G3):
		\begin{itemize}
			\item[(G1)] the multimap $G$ has nonempty, convex and weakly compact values; and the multimap $G(\cdot,x): I\multimap X$ is measurable for all $x\in X$.
			\item[(G2)] the multimap $G(t,\cdot): X\multimap X$ has strongly weakly closed graph, that is if $x_n\to x$ in $X$, $y_n\rightharpoonup y$ in $X$ with $y_n\in G(t,x_n)$, then $y\in G(t,x)$.
			\item[(G3)] there exists a function $\eta\in L^1(I,\mathbb{R}^+)$ such that
			\begin{equation}
				\norm{G(t,x)}=\sup\{\norm{y}_{X}: y\in G(t,x)\}\le \eta(t), ~\text{a.a.}~t\in I, x\in X.
			\end{equation} 
		\end{itemize}
	\end{remark}
	We now recall the measurable selection multimaps considered in \eqref{SF} and \eqref{Ssigma}. The following results discuss the well-posedness of these multimaps and some properties that are satisfied by them.

	The following Theorem says that the set $S_G(q)$ is nonempty for every $q\in C(I, X)$ and delivers some property.
	\begin{theorem}\cite[Lemma 5.3]{MR2976197}\label{Thm3.2}
		Suppose the multivalued map $G:I\times X \multimap X^*$ satisfies the hypotheses (G1)-(G3).
		Then, the following is true.
		\begin{itemize}
			\item[(i)] the multifunction $S_G$ has nonempty and weakly compact convex values.
			\item[(ii)] the multimap $S_G$ has strongly weakly closed graph in the following sense: suppose $g_n\in S_G(q_n)$ with $g_n\rightharpoonup g$ in $L^1(I, X^*)$, $q_n\to q$ in $C(I, X)$, then $g\in S_G(q)$. 
		\end{itemize} 
	\end{theorem}
	\begin{lemma}\label{stochastic measurable selection}\cite{MR3750671}
		Assume that the multimap $\Sigma$ satisfies the assumption ($\sigma$). Then
		\begin{itemize}
			\item[(i)] the set $S_{\Sigma}(q)$ is nonempty for every $q\in C_{\mathbb{F}}(I,L^2(\Omega,X))$.
			\item[(ii)] the multimap $S_{\Sigma}$ has a strongly-weakly closed graph with convex and weakly compact values.
		\end{itemize} 
	\end{lemma}

	\section{Approximate Controllability Results}
	As pointed out in the previous section, the mild (and hence weak solutions) to the differential inclusion problem \eqref{HV2} are the weak solutions of the hemivariational control problem \eqref{HV1}. Therefore, we focus on the differential inclusion problem \eqref{HV2} to deduce the approximate controllability of the hemivariational control problem \eqref{HV1}.

	The approximate controllability result for the Problem \eqref{HV2} is given by the following theorem.
	\begin{theorem}\label{MR}
		Assume the Hypotheses ($\mathcal{T}$), ($\sigma$) and (B) hold. Further, assume that the nonlinear multimap $G$ satisfies Hypotheses (G1)-(G3). Then, the problem \eqref{HV2} is approximately controllable in $I$. 
	\end{theorem}
	\begin{proof}
		Let $z\in X$ be the desired target we want to achieve in time $a$. Fix $\epsilon>0$. Define a multimap $\Gamma_{\epsilon}:  C_{\mathbb{F}}(I, L^2(\Omega, X))\multimap C_{\mathbb{F}}(I, L^2(\Omega, X))$ as follows: for $q\in  C_{\mathbb{F}}(I, L^2(\Omega, X))$, $y\in \Gamma_{\epsilon}(q)$ satisfies
		\begin{equation}\label{CY1}
			y(t)=T(t)x_0+\int_{0}^{t}T(t-s)[ g(s)+Bu(s)]ds+\int_{0}^{t}T(t-s)\sigma(s)dW(s),~ P~\text{a.s.},
		\end{equation}   
		where $g\in S_G(q)$, $\sigma\in S_{\Sigma}(q)$ and $u\in L^2_{\mathbb{F}}(I,U)$ is given by
		\begin{equation}\label{CY2}
			u(t)=B^*T^*(a-t)((\epsilon I+G(a))^{-1}Z(g)),
		\end{equation}
		where 
		\begin{equation}\label{CY33}
			Z(g)=z-T(a)x_0-\int_{0}^{a}T(a-s) g(s)ds-\int_{0}^{a}T(a-s)\sigma(s)dW(s).
		\end{equation}
		It is obvious that the process $T(\cdot)x_0$ is $\mathbb{F}$- adapted and the map $t\mapsto T(t)x_0$ is continuous. Thus, the process $T(\cdot)x_0$ belongs to $C_{\mathbb{F}}(I, L^2(\Omega, X))$. Also, by virtue of Proposition \ref{continuity of deterministic convolution} (given in the Appendix), the deterministic convolution $\displaystyle \int_{0}^{t}T(t-s) g(s)ds$ is well defined and $X$-valued, continuous and $\mathbb{F}$- adapted process. Moreover, by virtue of Holder's inequality, we compute
		\begin{align*}
			\sup_{t\in [0, a]}\mathbb{E}\norm{\int_{0}^{t}T(t-s) g(s)ds}_X^2	
			\le&\mathbb{E}\sup_{t\in [0, a]}\norm{\int_{0}^{t}T(t-s) g(s)ds}_X^2\\
			\le& aE \sup_{t\in [0, a]}\int_{0}^{t}\norm{T(t-s) g(s)}_X^2ds\\
			\le& aE \sup_{t\in [0, a]}\int_{0}^{t}\norm{T(t-s)}_{\mathcal{L}(X)}^2\norm{ g(s)}^2_Xds\\
			\le & aM^2 \mathbb{E}\sup_{t\in [0, a]}\int_{0}^{t}\norm{g(s)}_X^2ds\\
			\le & aM^2 \mathbb{E}\int_{0}^{a}\eta(s)ds<\infty.
		\end{align*}
		Thus, we conclude that the deterministic convolution $\int_{0}^{t}T(t-s) g(s)ds\in C_{\mathbb{F}}(I, L^2(\Omega, X))$. Similarly, the deterministic convolution $\int_{0}^{t}T(t-s)Bu(s)ds\in C_{\mathbb{F}}(I, L^2(\Omega, X))$. 
		
		We now show that the stochastic integral $\displaystyle \int_{0}^{t}T(t-s)\sigma(s)dW(s)$ is well-defined. It is obvious from the properties of $\sigma$ that the process $\sigma(s), ~s\in [0,t]$ is $\mathcal{L}^2_0$-valued, measurable process with
		\begin{equation}
			\mathbb{E}\int_{0}^{a}\norm{\sigma(s)}_{\mathcal{L}^2_0}^2ds\le \int_{0}^{a}\zeta(s)ds<\infty.
		\end{equation}
		Therefore, with the help of Proposition \ref{stochastic convolution estimate},
		\begin{align*}
			\sup_{t\in [0, a]} \mathbb{E}\norm{\int_{0}^{t}T(t-s)\sigma(s)dW(s)}_X^2\le & \mathbb{E}\sup_{t\in [0, a]} \norm{\int_{0}^{t}T(t-s)\sigma(s)dW(s)}_X^2\\
			\le & \mathbb{E}\int_{0}^{a}\norm{\sigma(s}_{\mathcal{L}^2_0}^2ds<\infty.
		\end{align*}
		Therefore, the stochastic integral $\displaystyle \int_{0}^{t}T(t-s)\sigma(s)dW(s)$ exists. Moreover, by Lemma \ref{stochastic convolution}, the stochastic integral is a $X$-valued continuous process. 
		Consequently, the multimap $\Gamma_{\epsilon}$ is well defined.

		We now prove that the multimap $\Gamma_{\epsilon}$ has a fixed point for each $\epsilon>0$. We apply the Fixed Point Theorem \ref{fixed} in this regard. We deduce this in several steps.\\

		\textbf{(STEP-I)} 	We now prove the map $\Gamma_{\epsilon}$ maps the set 
		\begin{equation}
			B_R=\{q\in  C_{\mathbb{F}}(I, L^2(\Omega, X)): \norm{q}_{ C_{\mathbb{F}}(I, L^2(\Omega, X))}\le R\},~R>0
		\end{equation}
		into the set $B_{N_0}$ for some $N_0$. Let $y\in \Gamma_{\epsilon}(q), q\in B_R$. Then $y$ satisfies \eqref{CY1}-\eqref{CY33}. We now estimate
		\begin{align*}
			&\sup_{t\in [0,a]}\mathbb{E}\norm{y(t)}_X^2
			\le
			\mathbb{E}\sup_{t\in [0,a]}\norm{y(t)}_X^2\\
			\le & 4\left(\mathbb{E}\sup_{t\in [0, a]}\norm{T(t)x_0}_X^2+	\mathbb{E}\sup_{t\in [0, a]}\norm{\int_{0}^{t}T(t-s) g(s)ds}_X^2\right.\\
			&\left.+\mathbb{E}\sup_{t\in [0, a]}\norm{\int_{0}^{t}T(t-s)Bu(s)ds}_X^2+	\mathbb{E}\sup_{t\in [0, a]}\norm{\int_{0}^{t}T(t-s)\sigma(s)dW(s)}_X^2\right)\\
			\le & 4\left(M^2E \norm{x_0}_X^2+ aM^2 \mathbb{E}\int_{0}^{a}\norm{ g(s)}_X^2ds\right.\\
			+& \left.aM^2\norm{B}^2 \mathbb{E}\int_{0}^{a}\norm{u(s)}_U^2ds+K_a\mathbb{E}\int_{0}^{a}\norm{\sigma(s)}_{\mathcal{L}^2_0}^2ds\right)\\
			\le & 4\left(M^2E \norm{x_0}_X^2+ aM^2 \mathbb{E}\int_{0}^{a}\eta(s)ds
			+aM^2\norm{B}^2 \mathbb{E}\int_{0}^{a}\norm{u(s)}_U^2ds+K_a\int_{0}^{a}\zeta(s)ds\right).
		\end{align*}
		Using the expression \eqref{CY2} 
		\begin{align*}
			\mathbb{E}\norm{u(s)}_U^2
			=&\mathbb{E}\norm{B^*T^*(a-s)\left((\epsilon I+G(a))^{-1}Z(g)\right)}^2\\
			\le &\mathbb{E}\left[ \frac{1}{\epsilon}M\norm{B}\norm{\left(\epsilon(\epsilon I+G(a))^{-1}Z(g)\right)}\right]^2\\
			= &\mathbb{E}\left[ \frac{1}{\epsilon}M\norm{B}\norm{\epsilon(\epsilon I+G(a))^{-1}Z(g)}\right]^2\\
			\le &\left[ \frac{1}{\epsilon}M\norm{B}\right]^2\mathbb{E}\norm{Z(g)}^2.
		\end{align*}
		In the last inequality, we use Lemma \ref{Lem3.1}.
		Also, we have,
		\begin{equation}\label{control estimate}
			\mathbb{E}\int_{0}^{a}\norm{u(s)}_U^2ds\le a \left[ \frac{1}{\epsilon}M\norm{B}\right]^2\mathbb{E}\norm{Z(g)}^2.
		\end{equation}
		Keeping the expression \eqref{CY33} in mind we estimate
		\begin{align}\label{estimate of zg}
			\mathbb{E}\norm{Z(g)}^2\le& 4\left[ \mathbb{E}\norm{z}^2+M^2\mathbb{E}\norm{x_0}^2+M^2a\int_{0}^{a}\eta(s)ds+K_a\int_{0}^{a}\zeta(s)ds\right].
		\end{align}
		With the help of estimates \eqref{control estimate} and \eqref{estimate of zg} finally we compute
		\begin{align*}
			&\left(\sup_{t\in [0,a]}\mathbb{E}\norm{y(t)}_X^2\right)
			\le \mathbb{E}\left(\sup_{t\in [0,a]}\norm{y(t)}_X^2\right)\\
			\le & 4\left(M^2E \norm{x_0}_X^2+ aM^2 \int_{0}^{a}\eta(s)ds
			+\frac{a^2M^4\norm{B}^4}{\epsilon^2}	\mathbb{E}\norm{Z(g)}^2+K_a\int_{0}^{a}\zeta(s)ds\right)\\
			\le & 4\left(M^2E \norm{x_0}_X^2+ aM^2\int_{0}^{a}\eta(s)ds+\frac{a^2M^4\norm{B}^4}{\epsilon^2}	4\left[ \mathbb{E}\norm{z}^2+M^2\mathbb{E}\norm{x_0}^2\right]\right.\\
			+& \left.\frac{a^2M^4\norm{B}^4}{\epsilon^2}	4\left[M^2a\int_{0}^{a}\eta(s)ds+K_a\mathbb{E}\int_{0}^{a}\zeta(s)ds\right]+K_a\int_{0}^{a}\zeta(s)ds\right)\\
			\le& K_1\mathbb{E}\norm{z}_X^2+K_2\mathbb{E}\norm{x_0}_X^2+K_3\norm{H}_{\mathcal{L}(X)}\int_{0}^{a}\eta(s)ds+K_4\int_{0}^{a}\zeta(s)ds\left(\frac{4a^2M^4\norm{B}^4}{\epsilon^2}+1\right),
		\end{align*}
		where
		\begin{equation}\label{K1}
			K_1=\frac{16a^2M^4\norm{B}^4}{\epsilon^2},
		\end{equation}
		\begin{equation}\label{K2}
			K_2=4M^2+\frac{16a^2M^6\norm{B}^4}{\epsilon^2},
		\end{equation}
		\begin{equation}\label{K3}
			K_3=4aM^2+\frac{16a^3M^6\norm{B}^4}{\epsilon^2},
		\end{equation}
		and
		\begin{equation}\label{K4}
			K_4=\frac{16a^2M^4\norm{B}^4}{\epsilon^2}K_a+K_a.
		\end{equation}
		From this we conclude that there exists $N_0$ (dependent on $\epsilon$) such that $\Gamma_{\epsilon}$ maps the set $B_{N_0}$ into $B_{N_0}$.\\

		\textbf{(STEP-II)} In this step we show that the multimap $\Gamma_{\epsilon}$ maps bounded sets in $ C_{\mathbb{F}}(I, L^2(\Omega, X))$ into relatively compact sets in $ C_{\mathbb{F}}(I, L^2(\Omega, X))$. Let $y_n\in \Gamma_{\epsilon}(q_n)$. 
		By the definition of the operator $\Gamma_{\epsilon}$, we obtain
		\begin{equation}\label{4.7}
			y_n(t)=T(t)x_0+\int_{0}^{t}T(t-s) g_n(s)ds+\int_{0}^{t}T(t-s)Bu_n(s)ds+\int_{0}^{t}T(t-s)\sigma_n(s)dW(s),~ t\in I, n\in \mathbb{N},
		\end{equation}
		where $g_n\in S_G(q_n)$, $\sigma_n\in S_{\Sigma}(q_n)$ and 
		\begin{equation}\label{CC1}
			u_n(t)=B^*T^*(a-t)\left((\epsilon I+G(a))^{-1}Z(g_n)\right),~ t\in I,
		\end{equation}
		where
		\begin{equation}\label{CC2}
			Z(g_n)=z-T(a)x_0-\int_{0}^{a}T(a-s) g_n(s)ds-\int_{0}^{t}T(t-s)\sigma_n(s)dW(s).
		\end{equation}
		By Hypothesis (G3), we conclude that the sequence $\{g_n\}_{n\in \mathbb{N}}\subset L_2^{\mathbb{F}}(I,X)$ is bounded. Hence we confirm that $g_n\rightharpoonup g$ in $L_2^{\mathbb{F}}(I,X)$.  Using the compactness of the semigroup $\{T(t)\}_{t\ge 0}$ we confirm that
		\begin{equation}\label{CC5}
			\int_{0}^{t}T(t-s) g_n(s)ds\to \int_{0}^{t}T(t-s) g(s)ds,~\text{in}~  C_{\mathbb{F}}(I, L^2(\Omega, X)).
		\end{equation}
		Now, the sequence $\{\sigma_n\}_{n\in \mathbb{N}}\subset L^2_{\mathbb{F}}(I,\mathcal{L}^2_0)$ is bounded and hence $\sigma_n\rightharpoonup \sigma^*$ in $L^2_{\mathbb{F}}(I,\mathcal{L}^2_0)$ up to a subsequence. Therefore, by Lemma \ref{stochastic conv}, we obtain 
		\begin{equation}
			\int_{0}^{t}T(t-s)\sigma_n(s)dW(s)\to \int_{0}^{t}T(t-s)\sigma^*(s)dW(s)~\text{in}~ C_{\mathbb{F}}(I, L^2(\Omega, X)).
		\end{equation}
		In particular,
		\begin{equation}
			\int_{0}^{a}T(a-s) g_n(s)ds\to \int_{0}^{a}T(a-s) g(s)ds~\text{as}~n\to \infty~\text{in}~L^2_{\mathbb{F}}(\Omega,X).
		\end{equation}
		Now, we estimate
		\begin{align*}
			\sup_{t\in [0,a]}\norm{\int_{0}^{t}T(t-s)[Bu_n(s)-Bu(s)]ds}_X^2\le& a \sup_{t\in [0, a]}\mathbb{E}\int_{0}^{t}\norm{T(t-s)}_{\mathcal{L}(X)}^2\norm{B}^2_{\mathcal{L}(U,X)}\norm{u_n(s)-u(s)}_U^2ds\\
			\le aM^2\norm{B}_{\mathcal{L}(U,X)}^2\mathbb{E}\int_{0}^{a}\norm{u_n(s)-u(s)}_U^2ds.
		\end{align*}
		Observing the expression of the control $u_n$ given in \eqref{CC1} we compute
		\begin{align*}\label{CC3}
			\mathbb{E}\int_{0}^{a}\norm{u_n(s)-u(s)}^2ds=&\mathbb{E}\int_{0}^{a}\norm{B^*T^*(a-s)[((\epsilon I+G(a))^{-1}Z(g_n))-(\epsilon I+G(a))^{-1}Z(g)]}^2ds.
		\end{align*}
		It is obvious that
		\begin{equation}\label{Cc4}
			Z(g_n)\to Z(g)=z-T(a)x_0-\int_{0}^{a}T(a-s) g(s)ds-\int_{0}^{t}T(t-s)\sigma^*(s)dW(s),~\text{in}~L^2(\Omega,X).
		\end{equation}
		By the continuity of the operator $(\epsilon I+G(a))^{-1}$ together with the convergence \eqref{Cc4} we have
		\begin{align*}
			((\epsilon I+G(a))^{-1}Z(g_n))
			\to ((\epsilon I+G(a))^{-1}Z(g))~\text{in}~L^2(\Omega,X).
		\end{align*}
		Therefore, we have
		\begin{align*}
			\mathbb{E}\int_{0}^{a}\norm{u_n(s)-u(s)}_U^2ds\le M^2\norm{B}^2&\mathbb{E}\int_{0}^{a}\norm{J((\epsilon I+G(a))^{-1}Z(g_n))-J(\epsilon I+G(a))^{-1}Z(g)]}^2ds.
		\end{align*}
		Keeping Lemma \ref{Lem3.1} we have
		\begin{align*}
			&\mathbb{E}\norm{(\epsilon I+G(a))^{-1}Z(g_n)}^2\\
			\le&E \norm{(\epsilon I+G(a))^{-1}Z(g_n)}^2\\
			\le&\frac{1}{\epsilon^2}\mathbb{E}\norm{z-T(a)x_0-\int_{0}^{a}T(a-s) g_n(s)ds-\int_{0}^{a}T(a-s)\sigma_n(s)dW(s)}^2\\
			\le&\frac{4}{\epsilon^2} \left[\mathbb{E}\norm{z}^2+M^2\mathbb{E}\norm{x_0}_X^2+M^2a\int_{0}^{a}\eta(s)ds+K_a\int_{0}^{a}\zeta(s)ds\right], t\in I.
		\end{align*}
		Therefore, an application of Lebesgue Dominated Convergence Theorem guarantee that
		\begin{equation}
			\mathbb{E}\int_{0}^{a}\norm{u_n(s)-u(s)}_U^2ds\to 0~\text{as}~n\to \infty,
		\end{equation}
		and hence 
		\begin{equation}\label{4uu}
			\int_{0}^{t}T(t-s)Bu_n(s)ds\to \int_{0}^{t}T(t-s)Bu(s)ds,~\text{uniformly in}~C_{\mathbb{F}}(I, L^2(\Omega, X)).
		\end{equation}
		Consequently, passing limit in \eqref{4.7} and observing the convergences \eqref{CC5} and \eqref{4uu}we obtain
		\begin{equation}
			y_n(t)\to y(t)=T(t)x_0+\int_{0}^{t}T(t-s)[ g(s)+Bu(s)]ds+\int_{0}^{t}T(t-s)\sigma^*(s)dW(s), ~\text{in}~ C_{\mathbb{F}}(I, L^2(\Omega, X)),
		\end{equation}
		which proves that the set $\Gamma_{\epsilon}(B_R)$ is relatively compact in $ C_{\mathbb{F}}(I, L^2(\Omega, X))$. In particular, the set $\Gamma_{\epsilon}(B_{N_0})$ is relatively compact in $ C_{\mathbb{F}}(I, L^2(\Omega, X))$. Let $Q^{\epsilon}=\overline{\operatorname{conv} \Gamma_{\epsilon}(B_{N_0})}$, then $Q^{\epsilon}$ is a compact convex subset of $ C_{\mathbb{F}}(I, L^2(\Omega, X))$ with $\Gamma_{\epsilon}(Q^{\epsilon})\subset Q^{\epsilon}$.\\
		
		\textbf{STEP-III} 
		It remains to prove that the multimap $\Gamma_{\epsilon}$ is upper semicontinuous. By virtue of  \cite[Theorem 1.1.12]{MR1831201}, it is sufficient to prove that the multimap $\Gamma_{\epsilon}$ has a closed graph. For this we consider $y_n\in \Gamma_{\epsilon}(q_n)$ with $q_n, y_n\in  C_{\mathbb{F}}(I, L^2(\Omega, X))$ and $y_n\to y, q_n\to q$ in $ C_{\mathbb{F}}(I, L^2(\Omega, X))$. We prove that $y\in \Gamma_{\epsilon}(q)$. By the definition of the multimap $\Gamma_{\epsilon}$ we obtain
		\begin{equation}
			y_n(t)=T(t)x_0+\int_{0}^{t}T(t-s)[ g_n(s)+Bu_n(s)]ds+\int_{0}^{t}T(t-s)\sigma_n(s)dW(s), P~\text{a.s},
		\end{equation}
		for some $g_n\in S_G(q_n), ~\sigma_n\in S_{\Sigma}(q_n)$. Proceeding similarly as in (\textbf{STEP-II}) we conclude that
		\begin{equation}
			y_n(t)\to y(t)=T(t)x_0+\int_{0}^{t}T(t-s)[ g(s)+Bu(s)]ds+\int_{0}^{t}T(t-s)\sigma^*(s)dW(s), ~\text{a.s},
		\end{equation}
		for some $g\in L^2_{\mathbb{F}}(I,X),~\sigma^*\in L^2_{\mathbb{F}}(I,\mathcal{L}^2_0)$. Here,  $u\in L^2_{\mathbb{F}}(I,U)$ is given by
		\begin{equation}
			u(t)=B^*T^*(a-t)((\epsilon I+G(a))^{-1}Z(g)),
		\end{equation}
		where 
		\begin{equation}\label{CY3}
			Z(g)=z-T(a)x_0-\int_{0}^{a}T(a-s) g(s)ds-\int_{0}^{a}T(a-s)\sigma^*(s)dW(s).
		\end{equation}
		By virtue of Theorems \ref{Thm3.2},\ref{stochastic measurable selection},  we conclude that $g\in S_G(q)$ and $\sigma^*\in S_{\Sigma}(q)$.
		Hence $y\in \Gamma(q)$ and consequently, the multimap $\Gamma_{\epsilon}$ has a closed graph.\\

		\textbf{(STEP-IV)} Following (\textbf{STEP II}) and (\textbf{STEP III}) we see that the multimap $\Gamma_{\epsilon}$ satisfies all the conditions of Theorem \ref{fixed}. Therefore, by Theorem \ref{fixed}, the map $\Gamma_{\epsilon}$ has a fixed point for each $\epsilon>0$, say $q_{\epsilon}$. By the definition of the multimap $\Gamma_{\epsilon}$ we have
		\begin{align}\label{q epsilon}
			q_{\epsilon}(t)=T(t)x_0+\int_{0}^{t}T(t-s)[ g_{\epsilon}(s)+Bu_{\epsilon}(s)]ds+\int_{0}^{t}T(t-s)\sigma_{\epsilon}(s)ds, t\in I.
		\end{align}
		In the above $g_{\epsilon}\in S_G(q_{\epsilon}), ~\sigma_{\epsilon}\in S_{\Sigma}(q_{\epsilon}$ and $u_{\epsilon}\in L^2(I,U)$ is given by
		\begin{equation}\label{u epsilon}
			u_{\epsilon}(t)=B^*T^*(a-t)\left((\epsilon I+G(a))^{-1}Z(g_{\epsilon})\right),
		\end{equation}
		and
		\begin{equation}
			Z(g_{\epsilon})=z-T(a)x_0-\int_{0}^{a}T(a-s) g_{\epsilon}(s)ds-\int_{0}^{t}T(t-s)\sigma_{\epsilon}(s)dW(s).
		\end{equation}

		\textbf{STEP-V}
		It remains to show that $\mathbb{E}\norm{q_{\epsilon}(a)-z}_X^2\to 0$ as $\epsilon\to 0$.
		
		By virtue of assumption (G3) and ($\sigma$) we conclude that $g_{\epsilon}\rightharpoonup g$ in $L^2_{\mathbb{E}}(I,X)$ and $\sigma_{\epsilon}\rightharpoonup \sigma^*$ in $L^2_{\mathbb{F}}(I,\mathcal{L}^2_0)$. Therefore, by virtue of Theorem \ref{stochastic conv}, we conclude that 
		\begin{align}\label{conv}
			z-T(a)x_0-\int_{0}^{a}T(a-s) g_{\epsilon}(s)ds-\int_{0}^{a}T(a-s)\sigma_{\epsilon}(s)dW(s)\to z-T(a)x_0-\int_{0}^{a}T(a-s) g(s)ds-\int_{0}^{a}T(a-s)\sigma^*(s)ds~\text{in}~L^2(\Omega,X).
		\end{align}
		Also,  recalling the definition of the map $G(a)$, using \eqref{q epsilon} and \eqref{u epsilon} we now estimate
		\begin{align*}
			q_{\epsilon}(a)
			=&T(a)x_0+\int_{0}^{a}T(a-s) g_{\epsilon}(s)ds+\int_{0}^{a}T(a-s)Bu_{\epsilon}(s)ds+\int_{0}^{a}T(a-s)\sigma_{\epsilon}(s)dW(s)\\
			=&T(a)x_0+\int_{0}^{a}T(a-s) g_{\epsilon}(s)ds+\int_{0}^{a}T(a-s)BB^*T^*(a-s)((\epsilon I+G(a))^{-1}Z(g_{\epsilon}))ds\\
			+&\int_{0}^{a}T(a-s)\sigma_{\epsilon}(s)dW(s)\\
			=&T(a)x_0+\int_{0}^{a}T(a-s) g_{\epsilon}(s)ds+G(a)((\epsilon I+G(a))^{-1}Z(g_{\epsilon}))+\int_{0}^{a}T(a-s)\sigma_{\epsilon}(s)dW(s)\\
			=&T(a)x_0+\int_{0}^{a}T(a-s) g_{\epsilon}(s)ds+\int_{0}^{a}T(a-s)\sigma_{\epsilon}(s)dW(s)\\
			+&(\epsilon I+G(a)-\epsilon I)(\epsilon I+G(a))^{-1}Z(g_{\epsilon}))+\int_{0}^{a}T(a-s)\sigma_{\epsilon}(s)dW(s)\\
			=&T(a)x_0+\int_{0}^{a}T(a-s) g_{\epsilon}(s)ds+z-T(a)x_0-\int_{0}^{a}T(a-s) g_{\epsilon}(s)ds-\int_{0}^{a}T(a-s)\sigma_{\epsilon}(s)dW(s)\\
			-&\epsilon(\epsilon I+G(a))^{-1}Z(g_{\epsilon})+\int_{0}^{a}T(a-s)\sigma_{\epsilon}(s)dW(s)\\
			=&z-\epsilon(\epsilon I+G(a))^{-1}\left(z-T(a)x_0-\int_{0}^{a}T(a-s) g_{\epsilon}(s)ds-\int_{0}^{a}T(a-s)\sigma_{\epsilon}(s)dW(s)\right)
		\end{align*}
		Now,
		\begin{equation}
			\mathbb{E}\norm{q_{\epsilon}(a)-z}_X^2\le \mathbb{E}\norm{\epsilon(\epsilon I+G(a))^{-1}\left(z-T(a)x_0-\int_{0}^{a}T(a-s) g_{\epsilon}(s)ds-\int_{0}^{a}T(a-s)\sigma_{\epsilon}(s)dW(s)\right)}_X^2\to 0~\text{as}~\epsilon\to 0^+.
		\end{equation}
		Therefore, by the convergence shown in \eqref{conv} together with Theorem \ref{Thm6.0}, we obtain $q_{\epsilon}(a)\to z$ as desired. The proof is completed.\\
		
	\end{proof}
	\begin{remark}\label{RM}
		Note that under the assumptions of Theorem \ref{MR}, we cannot replace Hypothesis (G3) and ($\sigma$3) with the following Hypotheses on $G=\partial F$ and $\Sigma$:
		\begin{itemize}
			\item[(G4)] there exists a function $\eta\in L^1(I,\mathbb{R}^+)$ and a constant $d>0$ such that
			\begin{equation}
				\norm{G(t,x)}=\sup\{\norm{y}_{X^*}: y\in G(t,x)\}\le \eta(t)+d\norm{x}, ~\text{a.a.}~t\in I, x\in X.
			\end{equation}
			\item[($\sigma$4)] there exists a function $\zeta\in L^1(I,\mathbb{R}^+)$ and a constant $d>0$ such that
			\begin{equation}
				\norm{\Sigma(t,x)}=\sup\{\norm{y}_{X^*}: y\in \Sigma(t,x)\}\le \zeta(t)+d\norm{x}, ~\text{a.a.}~t\in I, x\in X.
			\end{equation}  
		\end{itemize}
		
		This is because, if $q_{\epsilon}, \epsilon>0$ is a mild solution of the problem \eqref{HV2}, then $q_{\epsilon}$ has the form
		\begin{align}\label{MCM1}
			q_{\epsilon}(t)=T(t)x_0+\int_{0}^{t}T(t-s)[g_{\epsilon}(s)+Bu_{\epsilon}(s)]ds+\int_{0}^{t}T(t-s)\sigma_{\epsilon}(s)ds, t\in I.
		\end{align}
		In the above $g_{\epsilon}\in S_{G}(q_{\epsilon})$ and $\sigma_{\epsilon}\in S_{\Sigma}(q_{\epsilon})$. 
		In (\textbf{STEP-V}) of the proof of Theorem \ref{MR} we prove that $q_{\epsilon}(a)\to z$ as $\epsilon\to 0+$. In the proof we need to show that the sequence $\{g_{\epsilon}\}_{\epsilon>0}\subset L^2_{\mathbb{F}}(I, X)$ is bounded and the sequence $\{\sigma_{\epsilon}\}_{\epsilon>0}$ is bounded. Note that $q_{\epsilon}$ has the following bounds (see (\textbf{STEP-I}) of the proof of Theorem \ref{MR}).
		\begin{align}\label{bound}
			\norm{q_{\epsilon}(t)}\le& K_1\mathbb{E}\norm{z}_X^2+K_2\mathbb{E}\norm{x_0}_X^2+K_3
			\int_{0}^{a}\eta(s)ds+K_4\int_{0}^{a}\zeta(s)ds,
		\end{align}
		where $K_i$'s are given by \eqref{K1}-\eqref{K4}.
		Due to the $\frac{1}{\epsilon}$-term in the right hand side of \eqref{bound} the sequence of functions $\{g_{\epsilon}\}_{\epsilon>0}$ fails to be bounded or uniformly integrable under Hypothesis (G4). Therefore, we cannot conclude that $\{g_{\epsilon}\}_{\epsilon>0}\subset L^2(I,X)$ has a weakly convergent subsequence.
	\end{remark}
	\section{Application}
	Let $I=[0, a]$, $D=[0,\pi]\subset \mathbb{R}$ representing a rod, in which the temperature distribution is governed, and $(\Omega,\mathcal{F}, P)$ denotes a probability space. We consider a heat propagation problem with noise term given by the one-dimensional stochastic heat equation
	\begin{equation}\label{SE1}
		d y(\omega,t,\theta)+\partial_{\theta \theta}y(\omega,t,\theta)dt+\Sigma(t,y(\omega,t,\theta))dW(\omega,t,\theta)+Bv(\omega,t,\theta)dt\ni f(\omega,t,\theta) dt,~\omega\in \Omega, ~t\in I, ~\theta\in D;
	\end{equation}
	subject to the initial condition
	\begin{equation}\label{sE2}
		y(\omega,0,\theta)=x_0(\omega,\theta),~\omega\in \Omega, \theta\in D;
	\end{equation}
	and the Dirichlet boundary condition
	\begin{equation}\label{SE3}
		y(\omega,t,0)=y(\omega,t, \pi)=0,~\omega\in \Omega, ~t\in I.
	\end{equation}
	
	In the above process, $y(\omega, t, \theta)$ represents the system's temperature at every moment $t\in (0, a)$ and at every point $\theta\in D$. The nonlinear term $f$ represents the flux of the heat, and $W(t)$ is real standard one-dimensional Brownian motion in $L^2(D)$ over a probability space $(\Omega,\mathcal{F}, P)$. We also assume that 
	\begin{equation}
		\Sigma(t,y(\omega,t,\cdot))=[g_1(t,y(\omega,t,\cdot)),g_2(t,y(\omega,t,\cdot))],
	\end{equation}
	where $g_i:[0,a]\times X\to \mathcal{L}^2_0, ~i=1,2$ be such that
	\begin{itemize}
		\item[(g1)] $g_1$ is lower semicontinuous, and $g_2$ is upper semicontinuous;
		\item[(g2)] $g_1(t,y(\omega,t,\theta))\le g_2(t,y(\omega,t,\theta)),~t\in [0,a],~\omega\in \Omega,~\theta\in [0,\pi]$;
		\item[(g3)] there exists $\alpha\in L^{\infty}([0,1],\mathbb{R}^+)$ such that
		\begin{equation}
			\norm{g_i(t,y(\omega,t,\theta))}\le \alpha(t),~t\in I,~\omega\in \Omega,~\theta\in [0,\pi],~i=1,2.
		\end{equation}
	\end{itemize}
	
	We now prescribe two temperatures $s_1$ and $s_2$ with $s_1\le s_2$. Now we set up volume sources of heat so that $y(\omega,t,\theta)$, $\theta\in D $ deviates as little as possible from the interval $(s_1,s_2)$. These devices have limited power which remains in the closed intervals $[g_1,g_2]$ with $0\in [g_1, g_2]$. We assume that the heat flux satisfies the following relation with the Clarke subdifferential:
	\begin{equation}
		-f(\omega, t,\theta)\in \partial \Phi(\omega,t,y(\omega, t, \theta)),~\omega\in \Omega, \theta\in  D,
	\end{equation}
	where 
	\begin{equation}
		\partial \Phi(\omega, t, u)=\begin{cases}
			g_1, ~\text{if}~u\le s_1\\
			[g_1,0], ~\text{if}~u=s_1\\
			0, ~\text{if}~s_1\le u\le s_2\\
			[0, g_2], ~\text{if}~u=s_2\\
			g_2, ~\text{if}~u\ge s_2.
		\end{cases}
	\end{equation}
	is the Clarke subdifferential of $\Phi(\omega, t, \cdot)$ given below
	\begin{equation}
		\Phi(\omega, t, u)=\begin{cases}
			g_1(u-s_1), ~\text{if}~u<s_1\\
			0, ~\text{if}~s_1\le u\le s_2\\
			g_2(u-s_2), ~\text{if}~u\ge s_2.
		\end{cases}
	\end{equation}

	Now we come to the problem of controllability.\\
	Let $z(\theta)$ be a fixed temperature profile and $\epsilon>0$ be given. Is it possible that by changing the control term $v(\omega,t,\theta)$, we can regulate the temperature of the system in such a way that the temperature of the system at time $a$ satisfies 
	\begin{equation}
		\int_{\Omega}\int_D\abs{y(\omega,a, \theta)-z(\theta)}^2d\theta d\omega<\epsilon.
	\end{equation}
	We now approach by rewriting the control problem \eqref{SE1}-\eqref{SE3} as an abstract stochastic hemivariational problem  in the state space $X=L^2([0,\pi],\mathbb{R})$ and control space $U=L^2([0,\pi], \mathbb{R})$. Note that $X$ is a separable Hilbert space, and $U$ is a separable Hilbert space. To this aim, let us define
	\begin{equation*}
		q(t)(\omega,\theta)=y(\omega,t,\theta), u(t)(\omega, \theta)=v(\omega,t,\theta),~ t\in I,~ \theta\in [0,\pi],~\omega\in \Omega.
	\end{equation*}
	Now, let $x^*\in L^2(D)$. Then we obtain
	\begin{equation}\label{SE4}
		\langle d y(\omega,t,\theta)+ \partial_{\theta \theta}y(\omega,t,\theta)dt+\sigma(t,\theta)dW(\omega,t,\theta)+Bv(\omega,t.\theta)dt, x^*(\theta)\rangle=\left\langle f(\omega,t,\xi)d\xi dt, x^*(\theta)\right\rangle,~\theta\in [0,\pi],
	\end{equation}
	where $\sigma(t,\cdot)\in \Sigma(t,y(\omega,t,\cdot))$ for a.a. $t\in I$. 
	Integrating over $[0,\pi]$ we obtain from above
	\begin{equation}\label{SE5}
		\int_{0}^{\pi}\langle d y(\omega,t,\theta)+ \partial_{\theta \theta}y(\omega,t,\theta)dt+\sigma(t,\theta)dW(\omega,t,\theta)+Bv(\omega,t.\theta)dt, x^*(\theta)\rangle d\theta=\int_{0}^{\pi}\left\langle f(\omega,t,\theta) dt, x^*(\theta)\right\rangle d\theta.
	\end{equation}
	
	Now we can re-write  the equation \eqref{SE5} as
	which can be written as
	\begin{equation}\label{SE6}
		\langle d y(\omega,t,\cdot)+ \partial_{\theta \theta}y(\omega,t,\cdot)dt+\sigma(t,\cdot)dW(\omega,t,\cdot)+Bv(\omega,t.\cdot)dt, x^*(\cdot)\rangle_{X}=\left\langle f(\omega,t,\cdot) dt, x^*(\cdot)\right\rangle_{X} .
	\end{equation}
	
	Noting that $	q(t)(\omega,\theta)=y(\omega,t,\theta), u(t)(\omega, \theta)=v(\omega,t,\theta),~ t\in I,~ \theta\in [0,\pi],~\omega\in \Omega$ we obtain from above
	\begin{equation}\label{SE7}
		\langle dq(t)+ Aq(t)dt+\sigma(t)dW(t)+Bu(t)dt, x^*(\cdot)\rangle_{X}=\left\langle f(t) dt, x^*(\cdot)\right\rangle_{X} .
	\end{equation}
	where, $\sigma(t)\in \Sigma(t,q(t)),$ for a.a. $t\in I$ and the operator
	$A: D(A)\subset X\to X$ is defined as follows
	\begin{equation}
		Ay(\omega,t,\cdot)=\partial_{\theta \theta}y(\omega,t,\cdot),
	\end{equation}
	where
	\begin{equation}
		D(A)=H^2([0,\pi], \mathbb{R})\cap H^1_0([0,\pi], \mathbb{R}).
	\end{equation}

	The operator $B:U=L^2([0,\pi], \mathbb{R})\to X$ is the identity operator.
	Now we define
	\begin{equation}\label{Lambda0}
		F(t,x)=\int_{0}^{\pi}\Phi^0(t,\theta,x(\theta))d\theta,~ t\in I, x\in L^2([0,\pi],\mathbb{R}).
	\end{equation}
	Arguing as the proof of Theorem 3.47 \cite{MR2976197}, we derive the following Lemma.
	\begin{lemma}\label{Lemma 5.1}
		The function $F$ defined in \eqref{Lambda0} satisfies the following:
		\begin{itemize}
			\item[(F1)] $F(t,\cdot)$ is well defined and finite on $L^2([0,\pi],\mathbb{R})$ for a.a. $t\in I$.
			\item[(F2)] $F(\cdot, x)$ is measurable on $I$ for all $x\in L^2([0,\pi],\mathbb{R})$.
			\item[(F3)] $F(t,\cdot)$ is Lipschitz on bounded subsets of $L^2([0,\pi],\mathbb{R})$ a.a. $t\in I$.
			\item[(F4)] For all $x\in L^2([0,\pi],\mathbb{R}), v\in L^2([0,\pi],\mathbb{R})$ a.a. $t\in I$ we have
			\begin{equation}
				F^0(t,x;v)= \int_{0}^{\pi}\Phi^0(t,\theta, x(\theta);v(\theta))d\theta.
			\end{equation}
			In the above, $F^0(t,\theta,\cdot;\cdot)$  denotes the generalized Clarke directional derivative of the function $F(t,\theta,\cdot)$.			\item[(F5)] For all $x\in L^2([0,\pi],\mathbb{R})$ a.a. $t\in I$ we have
			\begin{equation}
				\partial F(t,x)= \int_{0}^{\pi}\partial \Phi(t,\theta,x(\theta))d\theta.
			\end{equation}
		\end{itemize}
	\end{lemma}
	By the definition of the Clarke subdifferential we obtain from $-f(\omega,t,\theta)\in \partial \Phi(\omega,t, y(\omega, t,\theta))$
	\begin{equation}
		\langle -f(\omega,t,\theta), x^*(\theta)\rangle \le \Phi^0(\omega, t, y(\omega,t,\theta);x^*(\theta)),~\theta\in [0,\pi].
	\end{equation}
	Integrating over $[0,\pi]$ we obtain
	\begin{equation}
		\langle -f(\omega,t,\cdot), x^*(\cdot)\rangle_{X} \le \int_{0}^{\pi}\Phi^0(\omega, t, y(\omega,t,\theta);x^*(\theta))d\theta,
	\end{equation}
	which implies
	\begin{equation}\label{S0}
		\langle - f(t), x^*(\cdot)\rangle_{X} \le \int_{0}^{\pi}\Phi^0(\omega, t, y(\omega,t,\theta);x^*(\theta))d\theta=F^0(t,q(t);x^*),~\text{a.s.}..
	\end{equation}
	Using \eqref{S0} together with \eqref{SE6}, the abstract reformulation of equation \eqref{SE1} may be given as the following semilinear evolution hemivariational inequalities in the Hilbert space $X=L^2([0,\pi],\mathbb{R})$:
	\begin{equation}\label{SE9}
		\langle dq(t)+ Aq(t)dt+\Sigma(t,q(t)))dW(t)+Bu(t)dt, x^*(\cdot)\rangle_{X}+F^0( t,q(t);x^*)\ge 0.
	\end{equation}
	Of course, the solutions to Problem \eqref{SE9} give rise to solutions for \eqref{HV1}. We now verify that all the Hypotheses of Theorem \ref{MR} hold. 
	
	The spectrum of the operator $A$ is given by $\{-n^2: n\in \mathbb{N}\}$. Then, for every $x\in D(A)$, the operator $A$ can be written as
	\begin{equation}
		Ax=\sum_{n=1}^{\infty}-n^2\langle x, w_n\rangle w_n,~ \langle x,w_n\rangle=\int_{0}^{\pi}x(\theta)w_n(\theta)d\theta,
	\end{equation}
	where $w_n(\theta)=\sqrt{\frac{2}{\pi}}\sin(n\theta)$ are the normalized eigenfunctions (with respect to the $L^2$ norm) of the operator $A$ corresponding to the eigenvalues $-n^2, n\in \mathbb{N}$. The strongly continuous semigroup $\{T(t)\}_{t\ge 0}$ generated by the operator $A$ is given by
	\begin{equation}
		T(t)x=\sum_{n=1}^{\infty}e^{-n^2t}\langle x, w_n\rangle w_n,~ x\in X.
	\end{equation}
	It is obvious that $\norm{T(t)}\le 1$ for all $t\ge 0$ and $T(t)$ is compact in $X$ for all $t>0$. Therefore, Hypothesis ($\mathcal{T}$) is satisfied. It is easy to prove that the linear control system is approximately controllable in $I$.
	Thus, we verify that Hypothesis ($\mathcal{T}$) and (B) holds. Also, by Lemma \ref{Lemma 5.1}, the function satisfies Hypotheses (1)-(6). Based on the properties (3) and (6), the function $F(t,x)$ satisfies all the conditions of Theorem \ref{MR}. Moreover, it is straightforward to check that the multimap $\Sigma$ satisfies Hypothesis ($\sigma$). Hence, by Theorem \ref{MR}, we conclude that the control system \eqref{SE1}-\eqref{SE3} is approximately controllable in $I$. 
	\section{Appendix}
	In this section, we define deterministic and stochastic convolution in Banach spaces and state some properties.
	\begin{theorem}
		Let $f$ be an $X$-valued stochastic process on $[0, a]$ such that the trajectories of $f$ are $P$- a.s. Bochner integrable; that means there exists $\bar{\Omega}\in \mathcal{F}$ with $P(\bar{\Omega})=1$ such that for each $\omega\in \bar{\Omega}$, the mapping $f(\cdot, \omega):[0, a]\to X$ is Borel measurable and $\int_{0}^{a}\norm{f(t,\omega)}_Xdt<\infty$. For each $t\in [0, a]$, define an $X$-valued stochastic process $u$ on $[0,t]$ by $u(r)=T(t-r)f(r), ~r\in [0,t]$. Then the trajectories of $u$ are $P$- a.s. Bochner integrable. 
	\end{theorem}
	\begin{corollary}
		If $f$ is an $X$- valued stochastic process on $[0, a]$ such that the trajectories of $f$ are $P$- a.s. Bochner integrable, then for each $t\in [0, a]$, the integral $\int_{0}^{t}T(t-r)f(r)dr$ exists $P$- a.s. Therefore, a process $\zeta$ defined by
		\begin{equation}
			\zeta(t)=\int_{0}^{t}T(t-r)f(r)dr, ~t\in [0, a]
		\end{equation}
		is an $X$- valued stochastic process on $[0, a]$ and is called the deterministic convolution.
	\end{corollary}
	\begin{proposition}\label{continuity of deterministic convolution}
		The deterministic convolution $\zeta$ is continuous, that is, the trajectories of $\zeta$ are $P$- a.s continuous on $[0, a]$; that means there exists $\bar{\Omega}\in \mathcal{F}$ with $P(\bar{\Omega})=1$ such that for each $\omega\in \bar{\Omega}$, the mapping
		\begin{equation}
			\zeta(\cdot,\omega):[0, a]\ni t\mapsto \zeta(t,\omega)=\int_{0}^{t}T(t-r)f(r,\omega)dr\in X
		\end{equation}
		is continuous on $[0, a]$.
	\end{proposition}
	\begin{lemma}\label{stochastic convolution}
		Let $\{T(t)\}_{t\ge 0}$ be a compact strongly continuous semigroup of bounded linear operators $T(t)$ acting on a separable reflexive Banach space $X$. Then the stochastic convolution $V:[0,a]\to L^2(\Omega, X)$ given by 
		\begin{equation}
			V(t)=\int_{0}^{t}T(t-s)\sigma(s)dW(s), ~t\in I,~~\sigma\in L^2_{\mathbb{F}}(I,\mathcal{L}^2_0)
		\end{equation}
		is uniformly continuous.
	\end{lemma}
	\begin{proof}
		Let $0\le t_1<t_2\le a$. We estimate
		\begin{equation}
			\mathbb{E}\norm{V(t_2)-V(t_1)}_X^2=\mathbb{E}\norm{\int_{0}^{t_2}T(t_2-s)\sigma(s)dW(s)-\int_{0}^{t_1}T(t_1-s)\sigma(s)dW(s)}_X^2.
		\end{equation}
		We consider the two cases:\\
		
		\textbf{case-I:} $t_1=0$ and $0<t_2\le a$. Then, thanks  to Theorem \ref{stochastic convolution estimate}, we estimate
		\begin{align*}
			\mathbb{E}\norm{V(t_2)-V(t_1)}_X^2=&\mathbb{E}\norm{\int_{0}^{t_2}T(t_2-s)\sigma(s)dW(s)}_X^2\\
			\le& \mathbb{E}\sup_{t\in [0,t_2]} \norm{\int_{0}^{t}T(t-s)\sigma(s)dW(s)}_X^2\\
			\le& \mathbb{E}\int_{0}^{t_2}\norm{\sigma(s)}_{\mathcal{L}^2_0}^2ds\to 0~\text{as}~t_2\to 0^+.
		\end{align*} 
		\textbf{Case-II:} Let $0<t_1\le t_2\le a$. Then we estimate
		\begin{align*}
			\mathbb{E}\norm{V(t_2)-V(t_1)}^2=&\mathbb{E}\norm{\int_{0}^{t_2}T(t_2-s)\sigma(s)dW(s)-\int_{0}^{t_1}T(t_1-s)\sigma(s)dW(s)}_X^2\\
			=&\mathbb{E}\norm{\int_{0}^{t_2-\epsilon}T(t_2-s-\epsilon)T(\epsilon)\sigma(s)dW(s)+\int_{t_2-\epsilon}^{t_2}T(t_2-s-\epsilon)T(\epsilon)\sigma(s)dW(s)\right.\\
				&\left.-\int_{0}^{t_1-\epsilon}T(t_1-s-\epsilon)T(\epsilon)\sigma(s)dW(s)-\int_{t_1-\epsilon}^{t_1}T(t_1-s-\epsilon)T(\epsilon)\sigma(s)dW(s)}_X^2\\
			=&\mathbb{E}\norm{\int_{0}^{t_1-\epsilon}[T(t_2-s-\epsilon)T(\epsilon)-T(t_1-s-\epsilon)T(\epsilon)]\sigma(s)dW(s)\right.\\
				+&\left.\int_{t_2-\epsilon}^{t_2}T(t_2-s-\epsilon)T(\epsilon)\sigma(s)dW(s)\right.\\
				&\left.+\int_{t_1-\epsilon}^{t_2-\epsilon}T(t_2-s-\epsilon)T(\epsilon)\sigma(s)dW(s)-\int_{t_1-\epsilon}^{t_1}T(t_1-s-\epsilon)T(\epsilon)\sigma(s)dW(s)}_X^2\\\
			\le& 4 \mathbb{E}\norm{\int_{0}^{t_1-\epsilon}[T(t_2-s-\epsilon)T(\epsilon)-T(t_1-s-\epsilon)T(\epsilon)]\sigma(s)dW(s)}^2_X\\
			+&4 \mathbb{E}\norm{\int_{t_2-\epsilon}^{t_2}T(t_2-s-\epsilon)T(\epsilon)\sigma(s)dW(s)}^2_X\\
			+& 4\mathbb{E}\norm{\int_{t_1-\epsilon}^{t_2-\epsilon}T(t_2-s-\epsilon)T(\epsilon)\sigma(s)dW(s)}^2_X+4\mathbb{E}\norm{\int_{t_1-\epsilon}^{t_1}T(t_1-s-\epsilon)T(\epsilon)\sigma(s)dW(s)}^2_X\\
			=& 4[V_1+ V_2+V_3+V_4].
		\end{align*}
		Using again Theorem \ref{stochastic convolution estimate} we obtain
		\begin{align*}
			V_1=&\mathbb{E}\norm{\int_{0}^{t_1-\epsilon}[T(t_2-s-\epsilon)T(\epsilon)-T(t_1-s-\epsilon)T(\epsilon)]\sigma(s)dW(s)}^2_X\\
			= & \mathbb{E}\norm{\int_{0}^{t_1-\epsilon}T(t_1-s-\epsilon)[T(t_2-t_1)T(\epsilon)-T(\epsilon)]\sigma(s)dW(s)}^2_X\\
			\le & \mathbb{E}\sup_{t\in [0, t_1-\epsilon]}\norm{\int_{0}^{t}T(t-s)[T(t_2-t_1)T(\epsilon)-T(\epsilon)]\sigma(s)dW(s)}^2_X\\
			\le& \mathbb{E}\int_{0}^{t_1-\epsilon}\norm{[T(t_2-t_1)T(\epsilon)-T(\epsilon)]\sigma(s)}^2_Xds\\\
			\le& \norm{T(t_2-t_1+\epsilon)-T(\epsilon)}^2 \mathbb{E}\int_{0}^{t_1-\epsilon}\norm{\sigma(s)}_{\mathcal{L}^2_0}^2ds.
		\end{align*}
		By assumption, the semigroup $\{T(t)\}_{t\ge 0}$ is compact for $t > 0$ and hence, from  \cite[Theorem 3.3, page 48]{MR710486} the operator
		$T(t)$ is continuous in the uniform operator topology for $\delta < t \le a$. Therefore, by using the continuity
		of $T(t)$ in the uniform operator topology we obtain $V_1\to 0$ as $t_2\to t_1$ and $\epsilon\to 0^+$.
		
		We now estimate 
		\begin{align*}
			V_2=&\mathbb{E}\norm{\int_{t_2-\epsilon}^{t_2}T(t_2-s-\epsilon)T(\epsilon)\sigma(s)dW(s)}^2_X\\
			=& \mathbb{E}\norm{\int_{t_2-\epsilon}^{t_2}T(t_2-s)\sigma(s)dW(s)}^p\\
			\le& \mathbb{E}\sup_{t\in [t_2-\epsilon, t_2]} \norm{\int_{t_2-\epsilon}^{t}T(t-s)\sigma(s)dW(s)}^2_X\\
			\le & E \int_{t_2-\epsilon}^{t_2}\norm{\sigma(s)}^2_{\mathcal{L}^2_0}ds\to 0~\text{as}~\epsilon\to 0^+, t_2\to t_1.
		\end{align*}
		In the same way we prove $V_3, V_4\to 0$ as $t_2\to t_1$ and $\epsilon\to 0^+$. This proves that the map $V$ is continuous.
	\end{proof}
	\begin{lemma}\label{Stochastic Convergence}
		Let the operator $\Xi: L^2_{\mathbb{F}}(I,\mathcal{L}^2_0)\to C_{\mathbb{F}}(I, L^2(\Omega, X))$ be such that
		\begin{equation*}
			\Xi(\sigma)(t)=\int_{0}^{t}T(t-s)\sigma(s)dW(s),~ t\in I,~\sigma\in L^2_{\mathbb{F}}(I,\mathcal{L}^2_0),
		\end{equation*}
		where $\{T(t)\}_{t\ge 0}$ is a compact, strongly continuous semigroup on a separable reflexive Banach space $X$. If $\{\sigma_n\}_{n\in \mathbb{N}}$ is bounded in $L^2_{\mathbb{F}}(I,\mathcal{L}^2_0)$ and uniformly integrable, then the sequence $\{\Xi(\sigma_n)\}_{n\in \mathbb{N}}$ is relatively compact in $C_{\mathbb{F}}(I, L^2(\Omega, X))$.
	\end{lemma}
	\begin{proof}
		We assume that the sequence $\{\sigma_n\}_{n\in \mathbb{N}}\subset L^2_{\mathbb{F}}(I,\mathcal{L}^2_0)$ is bounded and integrably bounded. We show the sequence $\{\Xi(\sigma_n)\}_{n\in \mathbb{N}}$ is relatively compact in $C_{\mathbb{F}}(I, L^2(\Omega, X))$. We use the Arzela-Ascoli Theorem to prove this. 
		
		We initially prove the sequence $\{\Xi(\sigma_n)\}_{n\in \mathbb{N}}\subset C_{\mathbb{F}}(I, L^2(\Omega, X))$ is pointwise relatively compact and equicontinuous. As in the proof of Lemma \ref{stochastic convolution} we prove that sequence $\{\Xi(\sigma_n)\}_{n\in \mathbb{N}}\subset C_{\mathbb{F}}(I, L^2(\Omega, X))$ is equicontinuous. Therefore, it remains to prove that the sequence $\{\Xi(\sigma_n)\}_{n\in \mathbb{N}}\subset C_{\mathbb{F}}(I, L^2(\Omega, X))$ is pointwise relatively compact.

		For $t=0$, it is trivial. 
		Take $t>0$. Fix $\delta>0$. We now define
		\begin{equation}
			Y_n(t)=\int_{0}^{t-\delta}T(t-s)\sigma_n(s)dW(s), ~n\in \mathbb{N}~t\in (0,a].
		\end{equation}
		Using the definition of semigroup $\{T(t)\}_{t\ge 0}$, we can also write it as
		\begin{equation}
			Y_n(t)=T(\delta)\int_{0}^{t-\delta}T(t-\delta-s)\sigma_n(s)dW(s)=T(\delta)Y_{n}^{\delta}(t),
		\end{equation}
		where
		\begin{equation}
			Y_n^{\delta}(t)=\int_{0}^{t-\delta}T(t-s-\delta)\sigma_n(s)ds~\text{for}~n\in \mathbb{N},~t\in (0,a].
		\end{equation}
		Then it is easy to see that the sequence $\{Y_n^{\delta}(t)\}_{n\in \mathbb{N}}$ is bounded in $L^2\Omega,X)$ for $t\in (0,a]$. Hence, by using the compactness of $T(t)$ for $t>0$, we can find a finite $z_i, i=1,2,\cdot\cdot\cdot p$ in $X$ such that
		\begin{equation}\label{total bound}
			\{Y_n(t)\}_{n\in \mathbb{N}}\subset \bigcup_{i=1}^{p} B\left(z_i,\frac{\epsilon}{2}\right).
		\end{equation}
		Using the properties of uniform integrability of the sequence $\{\sigma_n(\cdot)\}_{n\in \mathbb{N}}\subset L^2(I,\mathcal{L}^2_0)$ we compute
		\begin{align*}
			\mathbb{E}\norm{\Xi(\sigma_n)(t)-Y_n(t)}^2=&\mathbb{E}\norm{\int_{0}^{t}T(t-s)\sigma_n(s)dW(s)-\int_{0}^{t-\delta}T(t-s)\sigma_n(s)dW(s)}_X^2\\
			=& \mathbb{E}\norm{ \int_{t-\delta}^{t}T(t-s)\sigma_n(s)dW(s)}^2_X\\
			\le & \mathbb{E}\sup_{r\in [t-\delta, t]}\norm{ \int_{t-\delta}^{r}T(r-s)\sigma_n(s)dW(s)}^2_X\\
			\le &  K_a \mathbb{E}\int_{t-\delta}^{t}\norm{\sigma_n(s)}^2ds\to 0~\text{as}~\delta \to 0^+. 
		\end{align*} 
		Consequently, from \eqref{total bound} we infer that
		\begin{equation}
			\{\Xi(\sigma_n)(t)\}_{n\in \mathbb{N}}\subset \bigcup_{i=1}^{p}B(z_i,\epsilon).
		\end{equation}
		Thus for every $t\in I$, the sequence $\{\Xi \sigma_n(t)\}_{n\in \mathbb{N}}$ is relatively compact in $X$. Hence, the proof is concluded by employing the Arzela-Ascoli Theorem.
	\end{proof}
	
	\begin{corollary}\label{stochastic conv}
		Let the operator $\Xi: L^2_{\mathbb{F}}(I,\mathcal{L}^2_0)\to C_{\mathbb{F}}(I, L^2(\Omega, X))$ be such that
		\begin{equation*}
			\Xi(\sigma)(t)=\int_{0}^{t}T(t-s)\sigma(s)dW(s),~ t\in I,~\sigma\in L^2_{\mathbb{F}}(I,\mathcal{L}^2_0),
		\end{equation*}
		where $\{T(t)\}_{t\ge 0}$ is a compact strongly continuous semigroup on a separable reflexive Banach space $X$. Suppose the sequence $\{\sigma_n\}_{n\in \mathbb{N}}$ is bounded in $L^2_{\mathbb{F}}(I,\mathcal{L}^2_0)$ and uniformly integrable. If $\sigma_n$ converges weakly to $\sigma$ in $L^2_{\mathbb{F}}$ then the sequence $\{\Xi(\sigma_n)\}_{n\in \mathbb{N}}$ converges to $\Xi(\sigma)$ in $C_{\mathbb{F}}(I, L^2(\Omega, X))$.
	\end{corollary}
	\begin{proof}
		Suppose that the sequence $\sigma_n$ converges weakly to $\sigma$ in $L^2_{\mathbb{F}}(I,\mathcal{L}^2_0)$. We initially prove the sequence $\{\Xi(\sigma_n)(t)\}_{n\in \mathbb{N}}$ converges weakly to $\Xi(\sigma)(t)$ in $L^2(\Omega, X)$ for each $t\in [0,a]$. Fix $t\in [0, a]$. Let $x^*:L^2(\Omega, X)\to \mathbb{R}$ be a linear continuous operator. Then, it is easy to show that the operator
		\begin{equation}
			\sigma\to \int_{0}^{t}T(t-s)\sigma(s)dW(s)
		\end{equation}
		is linear and continuous operator from $ L^2_{\mathbb{F}}(I,\mathcal{L}^2_0))$ to $L^2(\Omega, X)$. Thus, we have the operator 
		\begin{equation}
			\sigma \mapsto x^*\circ \int_{0}^{t}T(t-s)\sigma(s)dW(s)
		\end{equation}
		is a linear and continuous operator from $L^2_{\mathbb{F}}(I,\mathcal{L}^2_0)$ to $\mathbb{R}$  for all $t\in [0,a]$. Further, from the definition of weak convergence, for every $t\in [0, a]$, we obtain
		\begin{equation}
			x^*\circ \int_{0}^{t}T(t-s)\sigma_n(s)dW(s)\to x^*\circ \int_{0}^{t}T(t-s)\sigma(s)dW(s).
		\end{equation}
		Thus, we proved that the sequence
		\begin{equation}
			\int_{0}^{t}T(t-s)\sigma_n(s)dW(s)\rightharpoonup \int_{0}^{t}T(t-s)\sigma(s)dW(s),
		\end{equation}
		in $L^2(\Omega, X)$. By Lemma \ref{Stochastic Convergence} we conclude that the sequence $\{\Xi(\sigma_n)\}_{n\in \mathbb{N}}$ is relatively compact in $C_F(I, L^2(\Omega, X))$. Therefore, we confirm that the sequence $\{\Xi(\sigma_n)\}_{n\in \mathbb{N}}$ converges to $\Xi(\sigma)$ in $C_F(I, L^2(\Omega, X))$. 
	\end{proof}
	\begin{lemma}\label{Determininstic Convergence}
		Let the operator $\Psi: L^2_{\mathbb{F}}(I,X)\to C_{\mathbb{F}}(I, L^2(\Omega, X))$ be such that
		\begin{equation*}
			\Psi(f)(t)=\int_{0}^{t}T(t-s)f(s)ds,~ t\in I,~f\in L^2_{\mathbb{F}}(I,X),
		\end{equation*}
		where $\{T(t)\}_{t\ge 0}$ is a compact strongly continuous semigroup on $X$. If $\{f_n\}_{n\in \mathbb{N}}$ is bounded in $L^2_{\mathbb{F}}(I,X)$, then the sequence $\{\Psi(f_n)\}_{n\in \mathbb{N}}$ is relatively compact in $C_{\mathbb{F}}(I, L^2(\Omega, X))$.
	\end{lemma}
	\bibliographystyle{elsarticle-num}
	\bibliography{sn-bibliography}
	
\end{document}